\documentclass[a4paper]{article}

\usepackage{amsmath,amssymb,mathtools,amsthm}
\usepackage{xcolor,graphicx,float,enumerate}
\graphicspath{{figures/}}

\usepackage[default,scale=1]{opensans}
\usepackage[T1]{fontenc}

\usepackage[utf8]{inputenc}
\DeclareUnicodeCharacter{00A0}{ }
\usepackage{geometry}
\newcommand{\ee}{\varepsilon}

\newcommand{\defeq}{{\coloneqq}}

\newcommand{\BUC}{{\mathrm{BUC}}}
\DeclareMathOperator{\N}{N}

\newcommand{\Rd}{{\mathbb R^d}}

\newcommand{\diver}{\nabla \cdot}

\DeclareMathOperator{\supp}{supp}

\newcommand*\diff{\mathop{}\!\mathrm{d}}

\usepackage[hidelinks]{hyperref}
\usepackage[nameinlink]{cleveref}

\newtheorem{theorem}{Theorem}[section]
\newtheorem{proposition}[theorem]{Proposition}%
\newtheorem{corollary}[theorem]{Corollary}%
\newtheorem{lemma}[theorem]{Lemma}%

\theoremstyle{definition}
\newtheorem{definition}[theorem]{Definition}%
\newtheorem{remark}[theorem]{Remark}%

\numberwithin{equation}{section}

\newcommand{\nc}{\normalcolor}
\usepackage[
url=false,
isbn=false,
maxcitenames=4,
giveninits,
maxbibnames=100,
doi=true,
date=year,
uniquename=false,
block=none]{biblatex}
\renewbibmacro{in:}{}
\DeclareFieldFormat[article]{citetitle}{#1}
\DeclareFieldFormat[article]{title}{#1}

\makeatletter
\renewcommand*{\@fnsymbol}[1]{\ensuremath{\ifcase#1\or \star \or \dagger\or \ddagger\or
		\mathsection\or \mathparagraph\or \|\or **\or \dagger\dagger
		\or \ddagger\ddagger \else\@ctrerr\fi}}
\makeatother

\begin{document}
\title{Vortex formation for a non-local interaction model  \\with  Newtonian repulsion and superlinear mobility }

\author{J.A. Carrillo\thanks{Mathematical Institute, University of Oxford, Oxford OX2 6GG, UK. Email: \href{mailto:carrillo@maths.ox.ac.uk}{carrillo@maths.ox.ac.uk}} %
\and D. Gómez-Castro\thanks {Mathematical Institute, University of Oxford, Oxford OX2 6GG, UK. Email: \href{mailto:gomezcastro@maths.ox.ac.uk}{gomezcastro@maths.ox.ac.uk}} \thanks{Instituto de Matemática Interdisciplinar, Universidad Complutense de Madrid} \and %
J.L. Vázquez\thanks{Departamento de Matemáticas, Universidad Autónoma de Madrid. \href {mailto:juanluis.vazquez@uam.es}{juanluis.vazquez@uam.es}}}
\maketitle

\vspace{-0.8cm}

\begin{abstract}
We consider density solutions for gradient flow equations of the form $u_t = \diver ( \gamma(u)  \nabla \N(u))$, where $\N$ is the Newtonian repulsive potential in the whole space $\Rd$ with the nonlinear convex mobility $\gamma(u)=u^\alpha$, and $\alpha>1$. We show that solutions corresponding to compactly supported initial data remain compactly supported for all times leading to moving free boundaries as in the linear mobility case  $\gamma(u)=u$. For linear mobility it was shown that there is a special solution in the form of a disk vortex of constant intensity in space $u=c_1t^{-1}$ supported in a ball that spreads in time like $c_2t^{1/d}$, thus showing a discontinuous leading front or shock. Our present results are in sharp contrast with the case of concave mobilities of the form $\gamma(u)=u^\alpha$, with $0<\alpha<1$ studied in \cite{Carrillo+GV+Vazquez2019}. There, we developed a well-posedness theory of viscosity solutions that are positive everywhere and moreover display a fat tail at infinity. Here, we also develop a well-posedness theory of viscosity solutions that in the radial case leads to a very detailed analysis allowing us to show a waiting time phenomena. This is a typical behavior for nonlinear degenerate diffusion equations such as the porous medium equation. We will also construct explicit self-similar solutions exhibiting similar vortex-like behaviour characterizing the long time asymptotics of general radial solutions under certain assumptions. Convergent numerical schemes based on the viscosity solution theory are proposed analysing their rate of convergence. We complement our analytical results with numerical simulations ilustrating the proven results and showcasing some open problems.
\end{abstract}

 \section{Introduction}
We are interested in the family of equations of the form
\begin{equation*}
	\begin{dcases}
		u_t = \diver ( \gamma(u)   \nabla v) & (0,+\infty) \times \mathbb R^d, \\
		-\Delta v = u & (0,+\infty) \times \mathbb R^d, \\
		u = u_0 & t = 0,
	\end{dcases}
\end{equation*}
where the function $\gamma(u)$ is called the mobility. They all correspond to gradient flows with nonlinear mobility of the Newtonian repulsive interaction potential in dimension $d\ge 1$
$$
\mathcal{F}(u)=\frac12 \int_{\mathbb R^d} \N(u) u \diff x\, ,
$$
with $\N(u)$ the Newtonian repulsive potential \cite{CLSS10}, as they can be written in the form
\begin{equation}\label{nmgf}
	\begin{dcases}
		u_t + \diver \left( \gamma(u) w\right)=0 & (0,+\infty) \times \mathbb R^d, \\
		w = -\nabla \frac{\delta \mathcal{F}}{\delta u} & (0,+\infty) \times \mathbb R^d.
	\end{dcases}
\end{equation}
We will consider nonnegative data and solutions. The linear case $\gamma(u)=u$ is well-known in the literature as a model for wave propagation in superconductivity or superfluidity, cf. Lin and Zhang \cite{Lin+Zhang2000}, Ambrosio, Mainini, and Serfaty \cite{Ambrosio+Mainini+Serfaty2011, Ambrosio+Serfaty2008}, Bertozzi, Laurent, and L\'eger \cite{BLL12}, Serfaty and Vazquez \cite{Serfaty+Vazquez2014}. In that case the theory leads to uniqueness of bounded weak solutions having the property of compact space support, and in particular there is a special solution in the form of a disk vortex of constant intensity in space $u=c_1t^{-1}$ supported in a ball that spreads in time like $c_2t^{1/d}$, thus showing a discontinuous leading front or shock. This vortex is the generic attractor for a wide class of solutions.

We want to concentrate on models with nonlinear mobility of power-like type $\gamma(u)=u^\alpha$, $\alpha>0$. The sublinear concave $0<\alpha<1$ range  was studied in our previous paper \cite{Carrillo+GV+Vazquez2019}. For nonnegative data the study provides a theory of viscosity solutions for radially symmetric initial data that are positive everywhere, and moreover display a fat tail at infinity. In particular the standard vortex of the linear mobility transforms into an explicit selfsimilar solution that reminds of the Barenblatt solution for the fast diffusion equation. A very detailed analysis is done for radially symmetric data and solutions via the corresponding mass function that satisfies a first-order Hamilton-Jacobi equation.

The present paper contains the rest of the analysis of power-like mobility for convex superlinear cases when $\gamma(u)=u^\alpha$, and $\alpha>1$. Again, we perform a fine analysis in the case of radially symmetric solutions by means of the study of the corresponding mass function. The theory of viscosity solutions for the mass function still applies. As for qualitative properties, let us stress that in this superlinear parameter range $\alpha>1$ solutions recover the finite propagation property and the existence of discontinuity fronts (shocks).  We analyse in detail how the stable asymptotic solution goes from the fat tail profile of the sublinear case $\alpha<1$ to the shock profile of the range $ \alpha>1$ when passing through the critical value $\alpha=1$.

Another important aspect of the well-posedness theory that we develop for viscosity solutions \nc with radially symmetric initial data, is that the classical approach based on optimal transport theory for equations of the form \eqref{nmgf} developed in \cite{AGS,DNS09,CLSS10} fail for convex superlinear mobilities as described in \cite{CLSS10} since the natural associated distance is not well-defined \cite{DNS09}. Therefore, our present results are the first well-posedness results for gradient flows with convex superlinear power-law mobilities, even if only for radially symmetric initial data. Finally, let us mention that we still lack of a well-posedness theory for gradient flow equations of the form \eqref{nmgf} with convex superlinear mobilities for general initial data, possibly showing that the vortex-like solutions are generic attractors of the flow.

We also highlight how different convex superlinear mobilities are with respect to the linear mobility case by showing the property of an initial waiting time for the spread of the support for radial solutions that we are able to characterize. Indeed, let $u_0$ be radial and supported in a ball: $\supp u_0 = \overline B_R$. We prove that there is finite waiting time at $r=R$ if and only if
\begin{equation}\label{aux}
\limsup_{ r \to R^- } \,\,\,\, (R - r)^{\frac{1-\alpha}{\alpha}}   \!\!\!\!\!\!\int \limits _{r < |x| < R} u_0  (x) \diff x  =C < +\infty.
\end{equation}
The waiting time phenomenon is typical of slow diffusion equations like the Porous Medium Equation \cite{Vazquez1984} or the $p$-Laplacian equation. In our class of equations it does not occur for the whole range $0<\alpha\le 1$. We are able to estimate the waiting time in terms of the limit constant $C$ in \eqref{aux}.

We combine the theory with the computational aspect: we identify suitable numerical methods and perform a detailed numerical analysis. Indeed, we construct numerical finite-difference convergent schemes and prove convergence to the actual viscosity solution for radially symmetric solutions based on the mass equation. By taking advantage of the connection to nonlinear Hamilton-Jacobi equations, we obtain monotone numerical schemes showing their convergence to the viscosity solutions of the problem with a uniform rate of convergence, see Theorem \ref{thm:convergence M to m} in constrast with the case of concave sublinear mobilities in \cite{Carrillo+GV+Vazquez2019}.

The paper is structured as follows. We start by constructing some explicit solutions and developing the theory for radially symmetric initial data by using the mass equation in Sections 2 and 3 respectively. We construct the general viscosity solution theory for radially symmetric initial data in Section 4 showing the most striking phenomena for convex superlinear mobilities: compactly supported free boundaries determined by sharp fronts and the waiting time phenomena. Section 5 is devoted to show the convergence of monotone schemes for the developed viscosity solution theory with an explicit convergence rate.  The numerical schemes constructed illustrate the sharpness of the waiting time phenomena result, and allow us to showcase interesting open problems in Section 6. 
A selection of figures illustrates a number of salient phenomena. We provide videos for some interesting situations as supplementary material in \cite{Videos}.

\section{Explicit solutions}
The aim of this section is to find some important explicit solutions of
\begin{equation}
\tag{P}
\label{eq:main equation}
\begin{dcases}
u_t = \diver ( u^\alpha   \nabla v) & (0,+\infty) \times \mathbb R^d \\
-\Delta v = u & (0,+\infty) \times \mathbb R^d \\
u = u_0 & t = 0
\end{dcases}
\end{equation}
Notice that, as in \cite{Carrillo+GV+Vazquez2019}, we still have that, for $C \ge 0$
\begin{equation}
\label{eq:Friendly Giant}
	\overline u(t) = ( C + \alpha t )^{-\frac 1 \alpha}
\end{equation}
is a solution of the PDE. The repulsion potential $v$ diverges quadratically at infinity. For $C = 0$ we recover the Friendly Giant solution, introduced in \cite{Carrillo+GV+Vazquez2019}.

\subsection{Self-similar solution}
The  algebraic calculations developed in \cite{Carrillo+GV+Vazquez2019}  still work, we  get self-similar solutions of the form
\begin{equation*}
	U(t,x) = t^{-\frac 1 \alpha} F( |x|t^{-\frac 1 {\alpha d}} ),
\end{equation*}
with the self-similar profile
\begin{equation}
	\label{eq:profile}
	F(|y|) = \left(  \alpha + \left(  \frac{ \omega_d |y|^d   }{ \alpha }  \right)^{-\frac{\alpha}{\alpha - 1}}  \right) ^{-\frac 1 \alpha}.
\end{equation}
Let us remark that, for $\alpha > 1$, we have $F(0) = 0$ and $F(+\infty) = \alpha ^{-\frac 1 \alpha}$ (see \Cref{fig:selfsimilar} for a sketch of the self-similar profiles depending on $\alpha$).
This is different to the case $0 <\alpha < 1$, where $F(0) $ is a positive constant and $F$ decays at infinity.  In our present case $\alpha > 1$, the self-similar solutions have infinite mass, whereas for $0 < \alpha < 1$ the self-similar solutions have finite mass.
\begin{figure}[H]
	\centering
	\includegraphics[width=.5\textwidth]{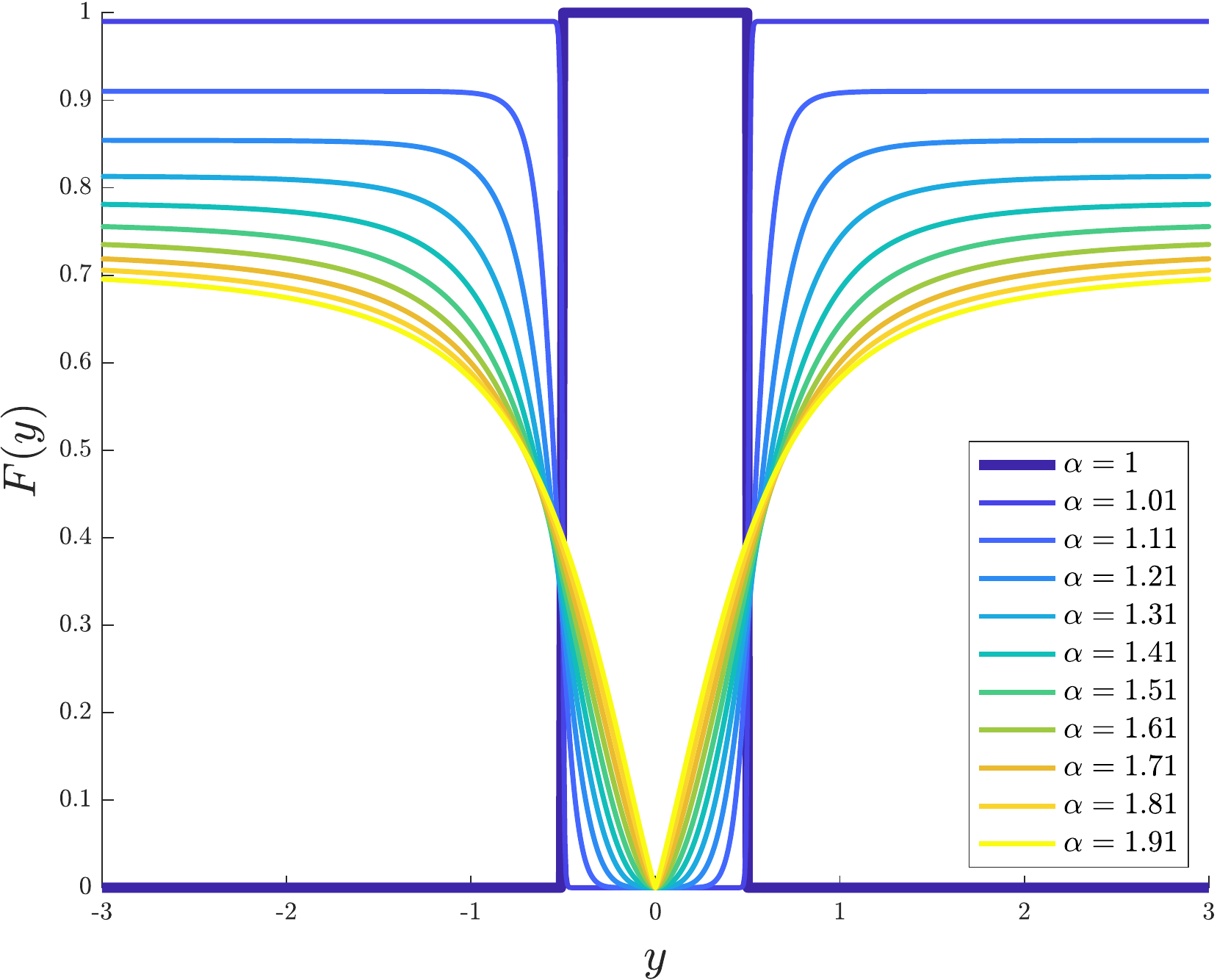}
	\caption{Self-similar profiles when $d=1$}
	\label{fig:selfsimilar}
\end{figure}
For $0 < \alpha < 1$ these self-similar give the typical asymptotic behaviour as $t \to + \infty$. For $\alpha > 1$ we will show this is no longer the case, for finite mass solutions.

\subsection{Vortices}

The vortex solutions defined as
\begin{equation}
	\label{eq:vortex}
	u(t,x) = \begin{cases}
	(u_0^{-\alpha} + \alpha t)^{-\frac 1 \alpha} & \omega_d |x|^d < S(t) = M / 	(u_0^{-\alpha} + \alpha t)^{-\frac 1 \alpha}  \\
	0 & \text{ otherwise}
	\end{cases}
\end{equation}
are local weak solutions of \eqref{eq:main equation}.

\begin{remark}
	Notice that, for $t \to - u_0^{-\alpha} / \alpha$, the vortex collapses to the Dirac delta of mass $M$.
\end{remark}
This solution was also constructed by characteristics and the Rankine-Hugoniot condition in \cite[Section 5.2]{Carrillo+GV+Vazquez2019}. However, in that case the Lax-Oleinik condition of incoming characteristics failed. In our present setting for $\alpha > 1$, this shock-type solutions are entropic, and we will prove that they are indeed viscosity solutions of the mass equation \eqref{eq:mass}.
We will prove that, for $\alpha > 1$, they now have significant relevance. In particular, they describe the asymptotic behaviour as $t \to +\infty$.
Notice that the the radius of the support $S(t)$ of this kind of solutions solves an equation of type
\begin{equation}
\label{eq:vortex RH}
\frac{\diff S}{\diff t} = M (u_0^{-\alpha} + \alpha t)^{-1 + \frac 1 \alpha}.
\end{equation}

There are also complementary vortex solutions:
\begin{equation}
	u(t,x) = \begin{dcases}
	0 & |x|^d < a\\
	(c_0^{-\alpha} + \alpha t)^{-\frac 1 \alpha} & |x|^d > a
	\end{dcases}
\end{equation}
which are stationary (and solve the mass problem \eqref{eq:mass} by characteristics). This type of solution belongs to a theory of solutions in $L^\infty$, but not $L^1$.

\medskip

\section{Mass of radial solutions}
In \cite{Carrillo+GV+Vazquez2019} we proved that the mass of a radial solution
\begin{equation*}
	m(t,r) = \int_{B_r} u(t,x) \diff x
\end{equation*}
written in volume coordinates $\rho = d \omega_d r^d$ is a solution of the equation
\begin{equation}
	\label{eq:mass}
	m_t + m (m_\rho)^\alpha = 0.
\end{equation}

\subsection{Characteristics for the mass equation}
In \cite{Carrillo+GV+Vazquez2019} we computed the generalised characteristics for the mass equation in the case of sublinear mobility, $\alpha < 1$.
The algebra for characteristics still works
\begin{equation}
	\label{eq:characteristics}
	\rho = \rho_0 + \alpha m(\rho_0) u_0(\rho_0)^{\alpha - 1} t\,,
\end{equation}
and the solutions behave like
\begin{equation}
	\label{eq:characteristics solutions}
	u(t,\rho) = (u_0(\rho_0)^{-\alpha} + \alpha t)^{-\frac 1 \alpha}\,,
\end{equation}
and
\begin{equation*}
	m(t, \rho) = m_0 (\rho_0) \left( 1 + \alpha u_0 (\rho_0)^\alpha t \right)^{1 - \frac 1 \alpha}\,.
\end{equation*}

\begin{remark}
	\label{rem:characteristics not level sets}
	Notice that the the generalised characteristics are not the level sets of $m$.
\end{remark}
\normalcolor

These solutions are well defined, for a given $u_0$, as long as the characteristics do not cross.

\begin{proposition}
	Let $u_0$ be non-decreasing and $C^1$. Then, there is a classical global solution of the mass equation, given by characteristics.
\end{proposition}
\begin{proof}
	Let $P_t(\rho_0) = \rho_0 + \alpha m(\rho_0) u_0(\rho_0)^{\alpha - 1} t$. Clearly $P_t (0) = 0$. If $u_0$ is non-decreasing, $\frac{d P_t}{d\rho_0} \ge 1$, and hence it is invertible. We construct
		\begin{equation}
		\label{eq:characteristics solutions formal}
	u(t,\rho) =
	\begin{dcases}
		 \left( \left( u_0(P_t^{-1} (\rho)) \right) ^{-\alpha} + \alpha t \right )^{-\frac 1 \alpha} &  \text{if } u_0(P_t^{-1} (\rho)) \ne 0, \\
		 0 & \text{if }  u_0(P_t^{-1} (\rho)) = 0.
		\end{dcases}
	\end{equation}
	It is immediate to see that $u$ is continuous and $C^1$.
\end{proof}

For $0 < \alpha < 1$, in \cite{Carrillo+GV+Vazquez2019} we developed a theory of classical solutions for non-increasing  initial data.  In that case, rarefaction fan tails appeared, which gave rise to classical solutions of the mass equation.
In our present case $\alpha > 1$, it seems that the good data is radially {non-decreasing}, but this is not possible in an $L^1 \cap L^\infty$ theory, unless a jump is introduced.

\subsection{The Rankine-Hugoniot condition}

We will prove in \Cref{sec:monotono non-decreasing with cut-off} that solutions with a jump, given by a Rankine-Hugoniot condition, will be the correct ``stable'' solutions. As in \cite{Carrillo+GV+Vazquez2019}, shocks (i.e. discontinuities) propagate following a Rankine-Hugoniot condition. If $S(t)$ is the position of the shock, we write the continuity of mass condition
\begin{equation*}
m(t, S(t)^- ) = m(t,S(t)^+).
\end{equation*}
Taking a derivative and applying the equation \eqref{eq:mass} we have that
\begin{equation}
\label{eq:RH}
\frac{\diff S}{\diff t} (t) = m(t, S(t)) \frac{u(t, S(t)^+)^\alpha - u(t,S(t)^- )^\alpha}{u(t, S(t)^+) - u(t,S(t)^- )}.
\end{equation}

\begin{remark}
	Notice that, in the case of the vortex the Rankine-Hugoniot condition determines precisely the support. In particular, we have $u(t,S(t)^+) = 0$ and $u(t,S(t)^-) = (u_0^{-\alpha} + \alpha t)^{-\frac 1 \alpha}$ so \eqref{eq:RH} is precisely \eqref{eq:vortex RH}. In fact, the vortex is simply a cut-off of the Friendly Giant with a free-boundary determined by the Rankine-Hugoniot condition.
\end{remark}

\subsection{Local existence of solutions by characteristics}

\begin{theorem}
	\label{thm:solution by characteristics}
	Let $0 \le u_0 \in L^1 (\mathbb R^n) \cap L^\infty (\mathbb R^n) $ be radial and such that $u_0^{\alpha - 1}$ is Lipschitz. Then, there exists a small time $T>0$ and a classical solution of the mass equation given by characteristics defined for $t \in [0,T]$.
\end{theorem}
\begin{proof}
	The solution given by \eqref{eq:characteristics} and \eqref{eq:characteristics solutions} is well defined as long as the characteristics cover the whole space, and do not cross. This is equivalent to $P_t(\rho_0) = \rho_0 + \alpha m(\rho_0) u_0(\rho_0)^{\alpha - 1} t$ being a bijection $[0,+\infty) \to [0,+\infty)$. Again, we construct \eqref{eq:characteristics solutions formal}.
	Since $P_t(0) = 0$, it suffices to prove that $\frac{\diff P_t}{\diff \rho_0} \ge c_0 > 0$. We take the derivative explicitly and find that
	\begin{align*}
		\frac{\diff P_t}{\diff \rho_0} &= 1 + \alpha t \left(  \frac{\diff m_0}{ \diff \rho_0 } u_0^{\alpha - 1} +m_0  \frac{\diff }{\diff \rho_0} (u_0^{\alpha - 1})\right) \\
		&= 1 + \alpha t \left(  u_0^{\alpha } + m_0 \frac{\diff }{\diff \rho_0} (u_0^{\alpha - 1})\right).
	\end{align*}
	Due to the hypothesis
	\begin{equation*}
		L \defeq \sup_{\rho_0 \ge 0} \left|   u_0^{\alpha } + m_0 \frac{\diff }{\diff \rho_0} (u_0^{\alpha - 1}) \right| < \infty,
	\end{equation*}
	and we have that $\frac{\diff P_t}{\diff \rho_0}  \ge 1 - \alpha t L $ which is strictly positive if $t \le T < 1 / ({\alpha L })$.
	Since $u_0^{\alpha - 1}$ is Lipschitz and bounded, then $1 - \alpha T L \le \frac{\diff P_t}{\diff \rho_0} \le C$ is Lipschitz. Hence, $P_t^{-1}$ is Lipschitz in $\rho_0$. Also, it is easy to see that $P_t^{-1}$ is continuous in $t$.
	Since it is immediate to check that $u$ is continuous by composition, we have that $m$ is of class $C^1$, and the proof is complete.
\end{proof}

\begin{corollary}[Waiting time]
	\label{cor:waiting time charact}
	Let $0 \le u_0 \in L^1 (\mathbb R^n) \cap L^\infty (\mathbb R^n) $ be radial and such that $u_0^{\alpha - 1}$ is Lipschitz. Then, there is a short time $T > 0$ such that, if $\supp u_0 \subset \overline{B_R}$, then any classical solution of the mass equation satisfies $\supp u(t,\cdot) \subset \overline{B_R}$ for $t < T$.
\end{corollary}

\begin{proof}
	Notice that, if $u_0$ is compactly supported, then outside the support the characteristics are given by $P_t (\rho_0) = \rho_0$. As long as the solution is given by characteristics, if $\supp u_0 \subset \overline{B_R}$ for $\rho > \omega_d R^d$, then $u(t,\rho) = 0$.
\end{proof}

\begin{remark}
	This effect of preservation of the support for a finite time is known as \emph{waiting time}.
	 In \Cref{sec:waiting time} we will show this holds true as long as the solution is smooth. In \Cref{sec:waiting time viscous} we show that this waiting time effect must be finite. This will lead us to show that solutions must lose regularity.
\end{remark}
\normalcolor

\begin{remark}
	Notice that higher regularity of $u_0$ is preserved by characteristics. Taking a derivative
	\begin{align*}
		\frac{\diff u}{\diff \rho} &= \left( \left( u_0(P_t^{-1} (\rho)) \right) ^{-\alpha} + \alpha t \right )^{-1 -\frac 1 \alpha} u_0^{-1-\alpha} \frac{\diff u_0}{\diff \rho_0} \frac{\diff P_t^{-1}}{\diff \rho.} \\
		&= \left(  1 + \alpha t u_0(P_t^{-1} (\rho))  ^\alpha  \right )^{-1 -\frac 1 \alpha}  \frac{\diff u_0}{\diff \rho_0} \frac{1} { \frac{\diff P_t}{\diff \rho_0}} \\
		&=\dfrac{	\left(  1 + \alpha t u_0(P_t^{-1} (\rho)) ) ^\alpha  \right )^{-1 -\frac 1 \alpha}  } {   1 + \alpha t \left(  u_0^{\alpha } + m_0 \frac{\diff }{\diff \rho_0} (u_0^{\alpha - 1})\right)   } \frac{\diff u_0}{\diff \rho_0}
	\end{align*}
	It is easy to see that, for small time, if $u_0$ is smooth enough, then $u$ is of class $\mathcal C^1$.
\end{remark}

\begin{remark}
	\label{rem:no sol by characteristics}
	The condition $u_0^{\alpha - 1}$ Lipschitz is sharp. Let us take, for $\varepsilon > 0$
	\begin{equation}
		\label{eq:no sol by characteristics}
		u_0 (\rho) = (c_0 - \rho)_+^{\frac {1-\varepsilon} {\alpha -1} }
	\end{equation}
	and let us show that characteristics cross for all $t>0$.
	Looking at the characteristics for $\rho_0 = c_0 -\delta$ with $\delta$ positive but small (so that $m_0 (\rho) > M/2$) we have that
	\begin{align*}
		\rho &= \rho_0 + \alpha m_0 (\rho) (c_0 - \rho_0)_+^{1 - \varepsilon} t  \\
		&\ge  c_ 0 - \delta + \alpha \frac{M} 2  \delta ^{1 - \varepsilon} t
	\end{align*}
	But for $\delta < \left(  \frac{\alpha M t}{2} \right)^{ \frac 1 {\varepsilon}}$ we have $\rho > c_0$. But this is not possible, since it must have crossed the characteristic $\rho = c_0$ coming from $\rho_0 = c_0$. No solutions by characteristics can exist.
\end{remark}

The crossing of characteristics will immediately lead to the formation of shock waves. The shock waves will be led by a Rankine-Hugoniot condition as above.

\begin{remark}
	As shown in \Cref{rem:no sol by characteristics}, with initial datum \eqref{eq:no sol by characteristics} we cannot expect solutions by characteristics. We could potentially paste solutions by characteristics on either side of a shock.
	We will show that this is the case, and we will show that solutions with bounded and compactly supported initial data will indeed produce a propagating shock at the end of their support, possibly with a waiting time (see the main results in \Cref{sec:viscosity}).
\end{remark}

\normalcolor

\subsection{Explicit Ansatz with waiting time}
\label{sec:waiting time}

For fixed mass $M$ and prescribed support of $u = m_\rho$
we can construct local classical solutions with waiting time $T$. We will prove that
\begin{equation}
	\label{eq:Ansatz}
	m(t, \rho)  = \left(  M^{\frac{\alpha}{\alpha - 1}} - \alpha^{\frac 1 {\alpha - 1}} \frac{(c_0 - \rho)_+^{\frac{\alpha}{\alpha - 1}}}{(T-t)^{\frac{1}{\alpha - 1}}}    \right)^{\frac{\alpha - 1}{\alpha}},
	\qquad  \text{if } t < T \text{ and } \rho > \left(c_0 - \alpha ^{\frac 1 \alpha} M (T-t)^{\frac 1 \alpha}\right)_+
\end{equation}
is a classical solution of $m_t + m_\rho^\alpha m = 0$.
We represent this function in \Cref{fig:Ansatz solution}. We will extend this function by zero for $ \rho \le \left(c_0 - \alpha ^{\frac 1 \alpha} M (T-t)^{\frac 1 \alpha}\right)_+$ to construct a viscosity subsolution of the mass equation.
\begin{figure}[H]
	\centering
	\includegraphics[width=.7\textwidth]{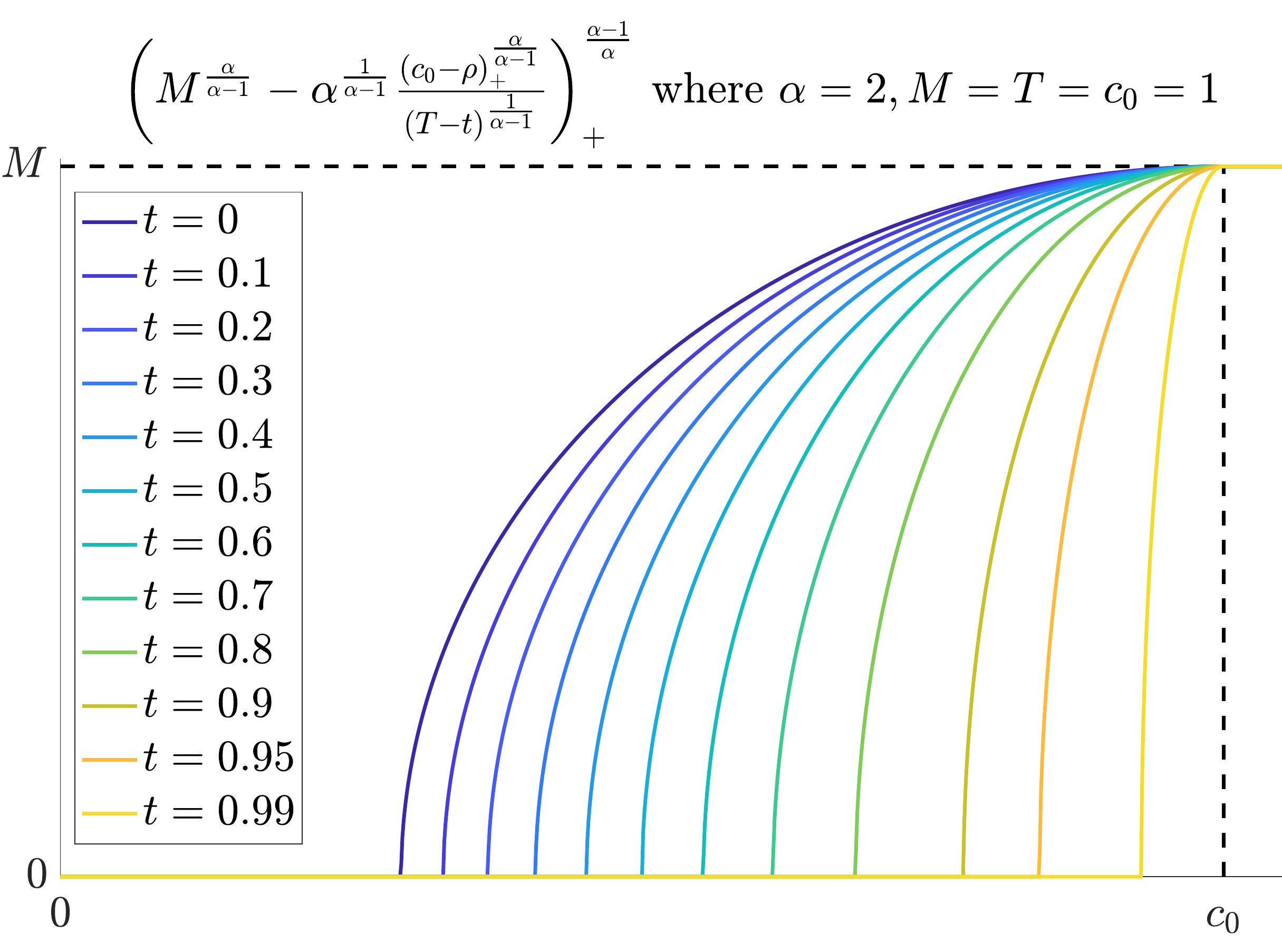}
	\caption{Ansatz solution}
	\label{fig:Ansatz solution}
\end{figure}

The intuition to construct this kind of explicit Ansatz in ``separated variables'' is well known in the context of nonlinear PDEs of power type (see, e.g., \cite{Vazquez1984}). One starts with a general formula of the type
\begin{equation*}
	m(t, \rho) =
	\left(  M^{\frac 1 {\beta}} - c \frac{(c_0 - \rho)^{\gamma}}{(T-t)^\delta}    \right)^{\beta}
\end{equation*}
and we match the exponents through the scaling properties of the equation. By taking $\beta = \frac{\alpha - 1}{ \alpha }$, $\gamma = \frac{\alpha}{\alpha - 1} > 1$ and $\delta = \frac{1}{\alpha-1}$ this gives
\begin{equation*}
		m_t + m m_\rho^\alpha = \left(  M^{\frac{1}{\beta}} - c \frac{(c_0 - \rho)^{\gamma}}{(T-t)^\delta}    \right)^{\beta - 1}  \frac{ (c_0 - \rho)^{\gamma}}{(T-t)^{\delta+1}}  \left[   - c \delta   \beta  + (c \gamma 	\beta)^\alpha   \right].
\end{equation*}
Hence, the sign is that of
$
	 - c \delta   \beta  + (c \gamma 	\beta)^\alpha$, which in our case is precisely $- \frac {c} \alpha + c ^\alpha .
$
\begin{remark}
We have also checked the following properties
\begin{enumerate}
	\item $m \in \mathcal C^1 $ in its domain of definition, since $\gamma \ge 1$, the matching at $\rho = c_0$ is guaranteed by the explicit formula for $m_\rho$.
	\item The domain of definition of $m$ depends on the value of {$T$}, and $\rho = 0$ may not be contained in the domain.
	
	\item It is easy to check that the function $u = m_\rho$ associate to the Ansatz satisfies
	\begin{equation}
		u(t,\rho) = \frac{c}{m(t,\rho)^{\frac 1 {\alpha-1}} (T-t)^{\frac{1}{\alpha-1}}} (c_0 - \rho)_+^{\frac 1 {\alpha - 1}}.
	\end{equation}
	Notice that $u^{\alpha-1}$ is a Lipschitz function of $\rho$.
	The easiest way to check this expression is by using
	\begin{equation*}
		m^{\frac { \alpha }{\alpha - 1}} = M^{\frac{\alpha}{\alpha - 1}} - c \frac{(c_0 - \rho)_+^{\frac{\alpha}{\alpha - 1}}}{(T-t)^{\frac{1}{\alpha - 1}}} ,
	\end{equation*}
	taking a derivative in $\rho$, and solving for $m_\rho$.
\item
	Conversely, notice that the condition $u_0^{\alpha-1}$ in \Cref{thm:solution by characteristics} is sharply satisfied by
	\begin{equation*}
		u_0 (\rho) \sim (c_0 - \rho)_+^{ \frac 1 {\alpha - 1}} .
	\end{equation*}
	If this is the behaviour of $u_0$ then close to $\rho = c_0$ we have that
	\begin{equation*}
		(m_0^{ \frac{\alpha }{\alpha  -1} } )_ \rho = \frac{\alpha}{\alpha-1} m_0^{ \frac{1}{\alpha  -1} } (m_0)_\rho = \frac{\alpha}{\alpha-1} m^{ \frac{1}{\alpha  -1} } u_0 \sim \frac{\alpha}{\alpha-1} M^{ \frac{1}{\alpha  -1} }  (c_0 - \rho)_+^{ \frac 1 {\alpha - 1}}
	\end{equation*}
	and, integrating in $[\rho,c_0]$
	\begin{equation*}
		M^{ \frac{\alpha }{\alpha  -1} } - m_0 ^{ \frac{\alpha }{\alpha  -1} } \sim M ^{ \frac{1 }{\alpha  -1} } (c_0 - \rho)_+^{ \frac{\alpha }{\alpha  -1} }.
	\end{equation*}
	Solving for $m$, we precisely recover the Ansatz
	\begin{equation*}
		m_0 \sim \left( M ^{ \frac{\alpha }{\alpha  -1} }  - M ^{ \frac{1 }{\alpha  -1} } (c_0 - \rho)_+^{ \frac{\alpha }{\alpha  -1} } \right)^{ \frac{\alpha - 1}{\alpha}}.
	\end{equation*}
\end{enumerate}
\end{remark}

\subsection{A change of variable to a Hamilton-Jacobi equation}

The change in variable $m = \theta^{\frac{\alpha - 1}{\alpha}}$ leads to the equation
\begin{equation}
\theta_t + \left( \frac{\alpha - 1}{\alpha } \right) ^{\alpha-1} \theta_\rho^\alpha = 0.
\end{equation}
This equation is of classical Hamilton-Jacobi form, and falls in the class studied by Crandall and Lions in the famous series of papers (see, e.g. \cite{Crandall1983}). Letting $w = \theta_\rho$ we recover a Burguer's conservation equation
\begin{equation}
\label{eq:problem w}
w_t + \left( \frac{\alpha - 1}{\alpha } \right) ^{\alpha-1} (w^\alpha)_\rho  = 0.
\end{equation}
The theory of existence and uniqueness of entropy solutions for this problem is well know (see, e.g. \cite{Benilan1996,Carrillo1999,Kruzkov1970} and related results in \cite{Escobedo+Vazquez+Zuazua1993,Lions+Souganidis+Vazquez2012}). Notice that the relation between $u$ and $w$ is somewhat difficult
\begin{equation*}
w(t, \rho) = \frac{\diff }{\diff \rho} \left[  \left(   \int_0^\rho u(t, \sigma) \diff \sigma   \right) ^{\frac {\alpha - 1}{\alpha}} \right].
\end{equation*}
\begin{remark}
	Notice that, for $\alpha < 1$ this change of variable does not make any sense since as $\rho \to 0$ we have that $m \to 0$ so $\theta \to +\infty$. We would be outside the $L^1 \cap L^\infty$ framework.
\end{remark}

\section{Viscosity solutions of the mass equation}
\label{sec:viscosity}
\subsection{Existence, uniqueness and comparison principle}

The Crandall-Lions theory of viscosity solutions developed in \cite{Carrillo+GV+Vazquez2019} for $\alpha < 1$, works also for the case $\alpha \ge 1$ without any modifications. Since $m_\rho = u \ge 0$ we can write the equation as
\begin{equation}
\label{eq:mass Burgers rho +}
m_t + (m_\rho)_+^\alpha m = 0.
\end{equation}
Then, the Hamiltonian $H(z,p_1,p_2) = (p_2)^\alpha z$ is defined and non-decreasing everywhere. We write the initial and boundary conditions
\begin{equation}
\label{eq:mass equation}
\begin{dcases}
m_t + (m_\rho)_+^\alpha m = 0 & t, \rho > 0 \\
m(t,0) = 0 & t > 0 \\
m(0, \rho) = m_0(\rho) & \rho > 0.
\end{dcases}
\end{equation}
The natural setting is with $m_0$ Lipschitz (i.e.\ $m_\rho = u \in L^\infty$) and bounded (i.e.\ $u \in L^1$).
We introduce the definition of viscosity solution for our problem and some notation.
\begin{definition}
	Let $f: \Omega \subset \mathbb R^n \to \mathbb R$. We define the Fréchet subdifferential and superdifferential
	\begin{align*}
	D^- u(x) &= \left \{ p \in \mathbb R^n : \liminf_{y \to x} \frac{u(y) - u(x) - p (y-x)}{|y-x|} \ge 0 \right\}\\
	D^+ u(x) &= \left \{ p \in \mathbb R^n : \limsup_{y \to x} \frac{u(y) - u(x) - p (y-x)}{|y-x|} \le 0 \right\}.
	\end{align*}
	
\end{definition}

\begin{definition}
	\label{defn:viscosity sub, super and solution}
	We say that a continuous function $m \in \mathcal C([0,+\infty)^2)$ is a:
	\begin{itemize}
		\item[$\bullet$] viscosity subsolution of \eqref{eq:mass equation} if
		\begin{equation}
		\label{eq:viscosity subsolution}
		p_1 + (p_2)_+^\alpha m(t, \rho) \le  0, \qquad \forall (t, \rho) \in \mathbb R_+^2 \text{ and } (p_1,p_2) \in D^+ m(t, \rho).
		\end{equation}
		and $m(0,\rho) \le m_0(\rho)$ and $m(t,0) \le  0$.

		\item[$\bullet$]  a viscosity supersolution of \eqref{eq:mass equation} if
		\begin{equation*}
		p_1 + (p_2)_+^\alpha m(t, \rho) \ge   0, \qquad \forall (t, \rho) \in \mathbb R_+^2 \text{ and } (p_1,p_2) \in D^- m(t, \rho).
		\end{equation*}
		and $m(0,\rho) \ge m_0(\rho)$ and $m(t,0) \ge 0$.
		
		\item[$\bullet$] a viscosity solution of \eqref{eq:mass equation} if it is both a sub and supersolution.
	\end{itemize}
\end{definition}
	
	The main results in \cite{Carrillo+GV+Vazquez2019} show the comparison principle and the well-posedness result for viscosity solutions of \eqref{eq:mass equation}.
	\begin{theorem}
		\label{thm:comparison principle m}
		Let $m$ and $M$ be uniformly continuous sub and supersolution of \eqref{eq:mass equation} in the sense of Definition~\ref{defn:viscosity sub, super and solution}. Then $m \le M$.
\end{theorem}
	We will denote by $\BUC$ the space of bounded uniformly continuous functions.
	\begin{theorem}	
		If $m_0 \in \BUC([0,+\infty))$ be non-decreasing such that $m_0(0) = 0$. Then, there exists a unique bounded and uniformly continuous viscosity solution.
		Furthermore, we have that
		\begin{equation}
		\label{eq:estimates on solution}
		\| m (t , \cdot) \|_{\infty} = \lim_{\rho \to +\infty} m(t,\rho) = \|m_0\|_{\infty}, \qquad \|m_\rho(t, \cdot) \|_\infty \le \| (m_0)_\rho \|_\infty.
		\end{equation}
		If $m_0$ is Lipschitz, then so is $m$.
	\end{theorem}

\subsection{The vortex}
\label{sec:vortex}
We next show that the mass associated to vortices \eqref{eq:vortex} (see \Cref{fig:mass vortex}) are viscosity solutions if and only if $\alpha \ge 1$. In \cite{Carrillo+GV+Vazquez2019} we showed that the mass associated to vortex solutions are not of viscosity type for $\alpha < 1$. Intuitively, this is another instance of degenerate diffusion versus fast diffusion-like behaviour, i.e. compactly supported versus fat tails.
\begin{theorem}
	\label{thm:vortex is viscosity sol}
	The mass associated to the  vortex solution \eqref{eq:vortex}
	\begin{equation}
	\label{eq:mass vortex}
	m(t,\rho) = \min  \{   (c_0^{-\alpha} + \alpha t )^{-\frac 1 \alpha} \rho, c_0 L  \}.
	\end{equation} is a viscosity solution for $\alpha \ge 1$.
\end{theorem}
\begin{figure}
	\centering
	\includegraphics[width = .45 \textwidth]{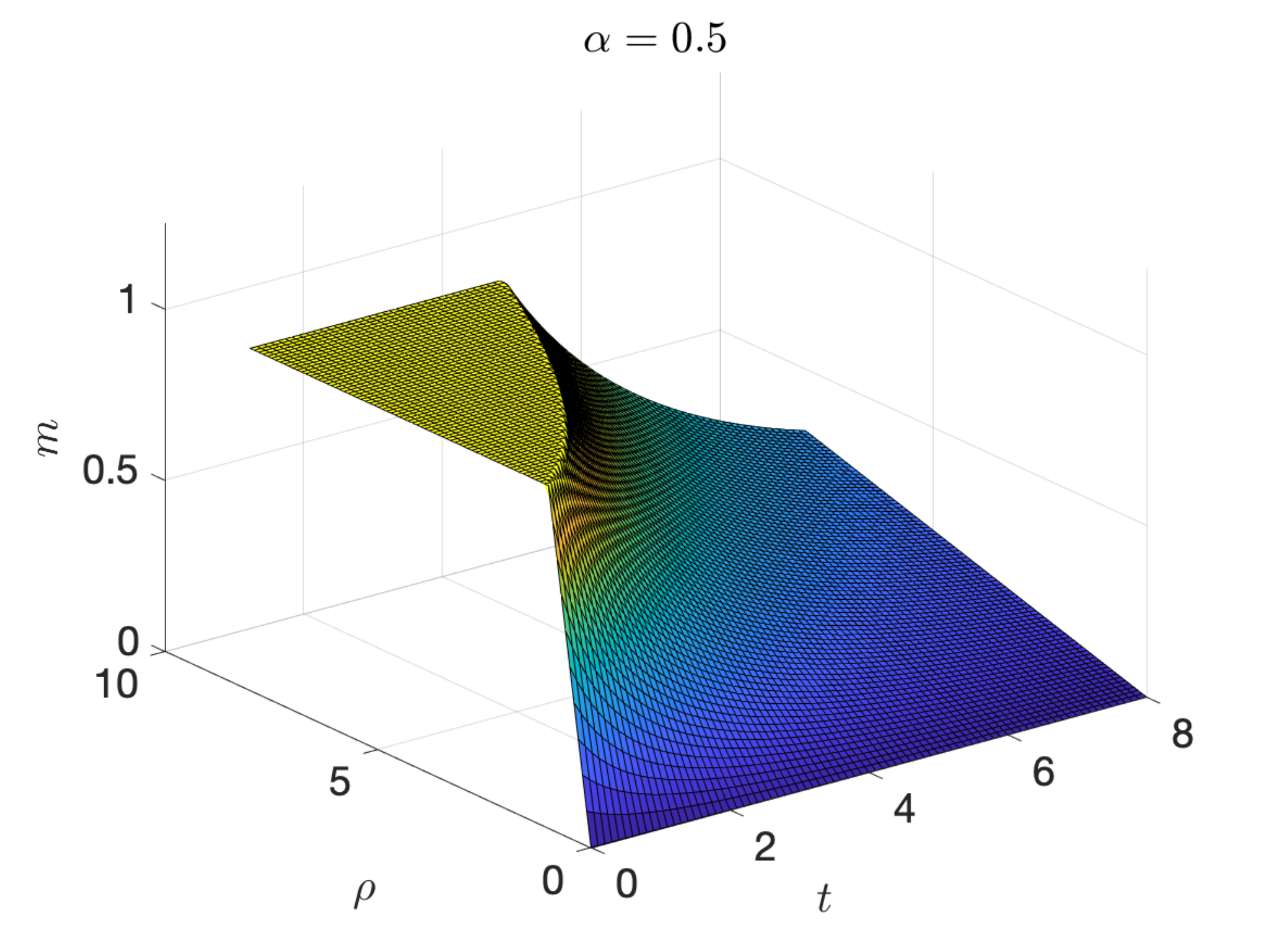}
	\includegraphics[width = .45 \textwidth]{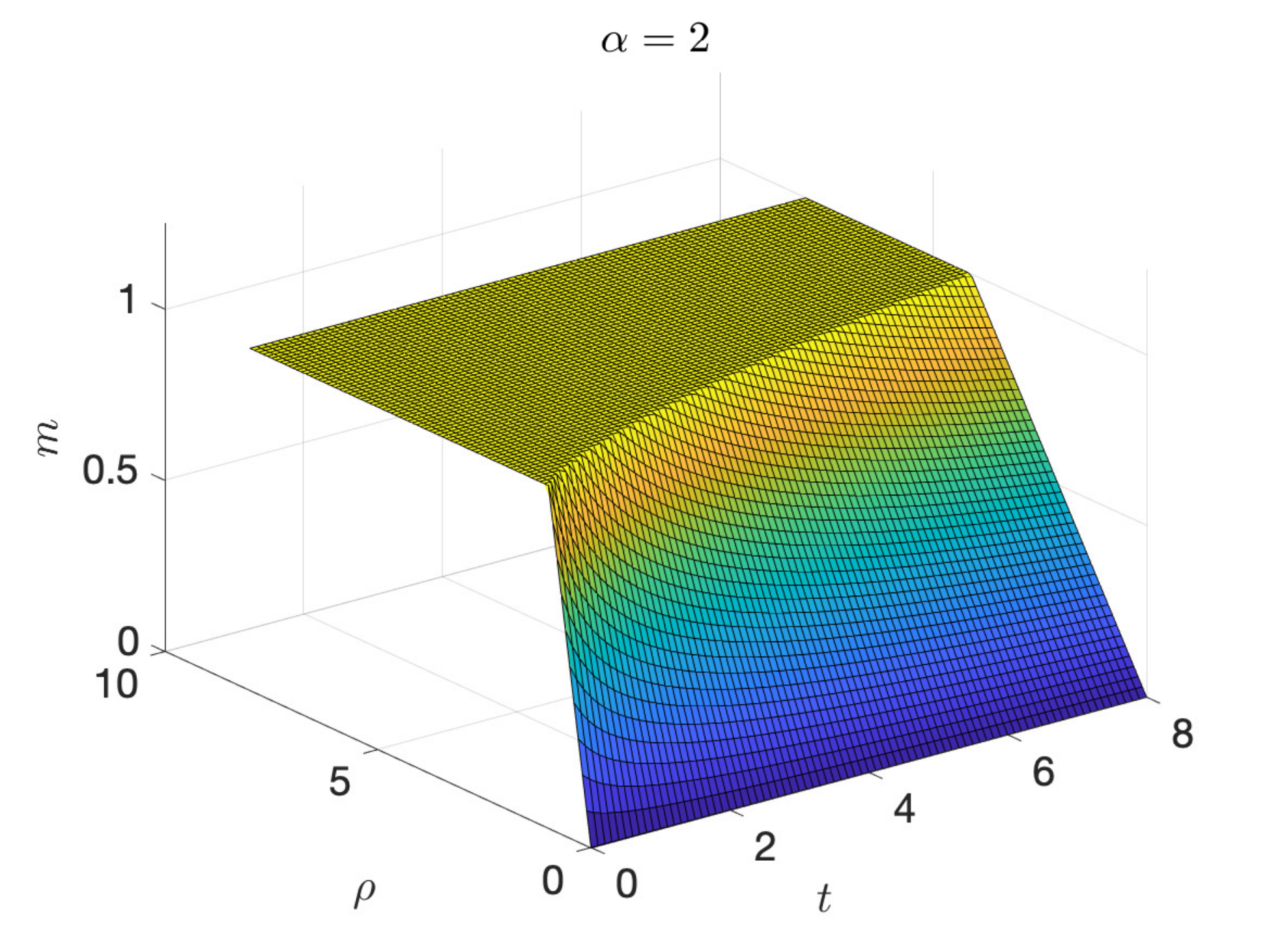}
	\caption{Representation of \eqref{eq:mass vortex} where $c_0 = L = 1$}
	\label{fig:mass vortex}
\end{figure}
\begin{proof}
	Let us fix a $t_0, \rho_0 > 0$. If we are not on the edge, i.e. $ (c_0^{-\alpha} + \alpha t_0 )^{-\frac 1 \alpha} \rho_0 \ne c_0 L$, then $m$ is a classical solution and it satisfies the viscosity formulation.
	
	Let us, therefore, look at a point such that
	\begin{equation}
	\label{eq:level set}
	(c_0^{-\alpha} + \alpha t_0 )^{-\frac 1 \alpha} \rho_0 = c_0 L.
	\end{equation}
	 It is clear that no smooth function $\varphi \le m$ can be tangent to $m$ at $(t_0, \rho_0)$, since the derivative of $m_\rho(t_0, \rho_0^-) > m_\rho(t_0, \rho_0^+)$.
	 Therefore, the definition of viscosity {supersolution} is immediately satisfied, since $D^- m (t_0,\rho_0) = \emptyset$.
	
	 To check that $m$ is a viscosity {subsolution}, let us take $(p_1,p_2) \in D^+ m(t_0, \rho_0)$.
	 Due to the explicit formula
	 $p_2 \in [0, (c_0^{-\alpha} + \alpha t)^{-\frac 1 \alpha}]$, and
	  since it is non-increasing in $t$, $p_1 \le 0$.
	
	 If $p_2= 0$ then, since $p_1 \le 0$, \eqref{eq:viscosity subsolution} is satisfied.
	Assume $p_2 \ne 0$. There exists $\varphi \in \mathcal C^1$  such that $\varphi \ge m$, $\varphi (t_0, \rho_0) = m (t_0, \rho_0) $ and $\varphi_t (t_0, \rho_0) = p_1$ and
	\begin{equation}
	\label{eq:estimate p_2}
		 \varphi_\rho (t_0, \rho_0) = p_2 \le (c_0^{-\alpha} + \alpha t)^{-\frac 1 \alpha}.
	\end{equation}
	
	Let us take a look at the level sets. Since $\varphi \ge m$, we have that $\{\varphi < c_0 L\} \subset \{ m < c_0L \}$. Therefore $\partial \{\varphi < c_0 L\}$, is contained in the region $\overline {\{ m < c_0 L \}} $. Since $\varphi$ and $m$ coincide at $(t_0, \rho_0)$, then $\partial \{\varphi < c_0 L\}$ and $\partial \{ m < c_0 L \}$ are tangent at $(t_0, \rho_0)$ (see \Cref{fig:level sets}).
	\begin{figure}[H]
		\centering
		\includegraphics[scale=1.2]{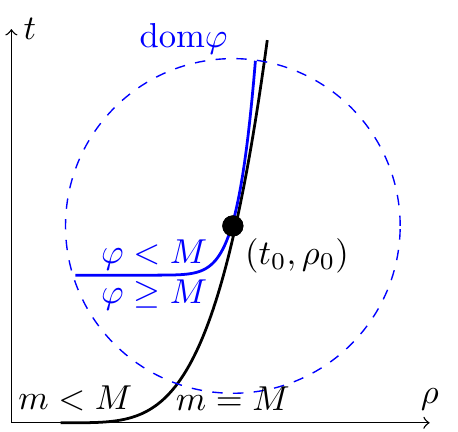}
		\caption{Relation between level sets. In the figure, $M = c_0 L$.}
		\label{fig:level sets}
	\end{figure}
	Since $p_2 > 0$, $\partial\{ \varphi < c_0 L \}$ can be parametrised by a curve $(t, \rho^* (t))$ defined for $t \in (t_0-\varepsilon, t_0 +\varepsilon)$ such that $\varphi(t, \rho^*(t)) = c_0 L$ and $\rho^* (t_0) = \rho_0$. Hence, taking a derivative and evaluating at $(t_0,\rho_0)$
	\begin{equation*}
	\varphi_t (t_0, \rho_0) + \frac{\diff \rho^*}{\diff t} (t_0) \varphi_\rho  (t_0, \rho_0) = 0.
	\end{equation*}
	On the other hand, $\partial \{ m < c_0 L \}$ can be parametrised as $(t, S(t))$ where
	\begin{equation*}
		S(t) = c_0L (c_0^{-\alpha} + \alpha t)^{ \frac 1 \alpha}
	\end{equation*}
	and hence
	\begin{equation}
	\label{eq:new 2.4}
		\frac{\diff S}{\diff t} = c_0L (c_0^{-\alpha} + \alpha t)^{-1 + \frac 1 \alpha}.
	\end{equation}
		Notice that the sign of $\alpha - 1$ decides the convexity or concavity of the matching curve. Since the level sets are tangent at $(t_0, \rho_0)$, the derivatives coincide and we have
	\begin{equation*}
	\frac{\diff \rho^*}{\diff t} (t_0) = \frac{\diff S}{\diff t} (t_0)  = c_0 L (c_0^{-\alpha} + \alpha t_0)^{-1 + \frac 1 \alpha}.
	\end{equation*}
	Hence, we have that
	\begin{equation*}
			p_1 =	\varphi_t (t_0, \rho_0) =  -  \frac{\diff \rho^*}{\diff t} (t_0) \varphi_\rho  (t_0, \rho_0) = -c_0 L (c_0^{-\alpha} + \alpha t)^{-\frac 1 \alpha} p_2.
	\end{equation*}
	Applying \eqref{eq:estimate p_2} and that $\alpha \ge 1$ we have that
	\begin{align*}
		p_1 + m(t_0, \rho_0) p_2^\alpha & =  -c_0 L (c_0^{-\alpha} + \alpha t)^{-\frac 1 \alpha} p_2 + c_0 L p_2^\alpha \\
		&= c_0 L p_2 \left(  p_2^{\alpha - 1} -  (c_0^{-\alpha} + \alpha t)^{-\frac 1 \alpha}  \right)\\
		&\le 0,
	\end{align*}
	which is precisely \eqref{eq:viscosity subsolution}.
	This completes the proof.
\end{proof}

\begin{remark}	
	
	\label{rem:one Delta data}
	The notion of viscosity solution can be extended, by approximation, to cover non-negative finite measures as initial data $u_0$.
	Notice that these vortex solutions concentrate as $t \searrow  -c_0^{-\alpha} / \alpha$ to the Heaviside function $m(t, \rho) \to c_0 L  H_0 (\rho)$. As we point out above, this proves that if
	\begin{equation}
		u_0 = M \delta_0 \quad (\text{i.e. } m_0 = M H_0 ),
	\end{equation}
	then the solution is a cut-off of the Friendly Giant \eqref{eq:Friendly Giant}
	\begin{equation*}
		u(t,\rho) = \begin{dcases}
		(\alpha t)^{-\frac 1 \alpha}  & \rho < M (\alpha t)^{\frac 1 \alpha} \\
		0  & \rho > M (\alpha t)^{\frac 1 \alpha}
		\end{dcases},
	\end{equation*}
	and hence
	\begin{equation}
		m(t,\rho) = \begin{dcases}
		(\alpha t)^{-\frac 1 \alpha} \rho  & \rho \le  M (\alpha t)^{\frac 1 \alpha} \\
		M  & \rho > M (\alpha t)^{\frac 1 \alpha}
		\end{dcases}.
	\end{equation}
	For every $\varepsilon>0$ this is a viscosity solution of \eqref{eq:mass equation} in $\mathcal C ((\varepsilon,+\infty) \times \mathbb R_+)$. Notice that $m$ is of self-similar form
	\begin{equation}
		{m\left( t ,  \rho \right) } = {M G\left (\frac{\rho}{M(\alpha t)^{\frac 1 \alpha}} \right)}  ,
		\qquad \text{where }
		G( y   ) = \begin{dcases}
		y  & y\le 1 \\
		1 & y > 1\\
		\end{dcases}.
	\end{equation}
	
	In fact, by translation invariance, it is possible to show that if $u_0 = M \delta_{c_0}$, i.e.
	\begin{equation*}
		m_0 (\rho) = M H _{c_0} (\rho),
	\end{equation*}
	then
	\begin{equation}
		m(t, \rho) = \begin{dcases}
		0 & \rho < c_0 \\
		(\alpha t)^{-\frac 1 \alpha} (\rho - c_0) & \rho \in [c_0, c_0 + M (\alpha t)^{\frac 1 \alpha}) \\
		M & \rho > c_0 + M (\alpha t)^{\frac 1 \alpha}
		\end{dcases}.
	\end{equation}
	is a viscosity solution of the mass equation defined for $t > 0$. With self-similar form
	\begin{equation}
	{m\left( t ,  \rho \right) } = {M G\left (\frac{\rho - c_0 }{M(\alpha t)^{\frac 1 \alpha}} \right)}  ,
	\qquad \text{where }
	G( y   ) = \begin{dcases}
	0 & y \le 0 \\
	y  & y\in (0,1) \\
	1 & y > 1\\
	\end{dcases}.
	\end{equation}
\end{remark}

\subsection{Monotone non-decreasing data with final cut-off}
\label{sec:monotono non-decreasing with cut-off}

As in \cite{Carrillo+GV+Vazquez2019}, the theory of existence and uniqueness is written in terms of $m$, but we take advantage of the intuition from the conservation law \eqref{eq:main equation} for $u$, to construct explicit solutions  through characteristics. Notice that taking a derivative in \eqref{eq:mass equation} we can write
\begin{equation*}
	u_t + (u^\alpha m)_\rho = 0.
\end{equation*}
Afterwards, we check that the constructed solution fall into our viscosity theory for $m$.

Since it is suggested by \eqref{eq:characteristics} that characteristics do not cross if $u_0$ is non-decreasing, let us look for solutions with initial data
\begin{equation}
\label{eq:jump sol u_0}
u_0 (\rho) = \begin{dcases}
\widetilde u_0 (\rho) & \rho < S_0 \\
0 & \rho > S_0.
\end{dcases}, \qquad \text{ where } \widetilde u_0 \text { in continuous and non-decreasing in } [0,S_0].
\end{equation}
When $m_0$ and $u_0$ are non-decreasing, it clear that characteristics with foot on $\rho_0 \in [0,S_0]$ do not cross.
If $\widetilde u_0 \not \equiv 0$, then $u_0(S_0^-) = \widetilde u_0 (S_0) > 0$ and there is a shock starting at $S_0$ which will propagate following the Rankine-Hugoniot condition \eqref{eq:RH}.

We construct the viscosity solution as follows. The characteristic of foot $\rho_0 = S_0$ is precisely
\begin{equation*}
\overline S (t) = S_0 + \alpha M \widetilde u_0 (S_0) ^{\alpha - 1} t.
\end{equation*}
For $\rho < \overline S(t)$ we can go back through the characteristics with $P_t$ defined above. Let us define
\begin{equation}
\widetilde u(t,\rho) = \begin{dcases}
0 &  \text{if } \widetilde u_0 (P_t^{-1} (\rho)) = 0 \\
(\widetilde u_0(P_t^{-1} (\rho))^{-\alpha} + \alpha t) ^{-\frac 1 \alpha} &   \text{if } \widetilde u_0 (P_t^{-1} (\rho)) > 0 \text{ and } \rho < \overline S(t) \\
0 & \text{if } \rho > \overline S (t)
\end{dcases},
\end{equation}
The shock is given by
\begin{equation}
\label{eq:jump sol RH}
\begin{dcases}
\frac{\diff S}{\diff t} = M \widetilde  u(t, S(t))^{\alpha - 1} \\
S(0) = S_0
\end{dcases}.
\end{equation}
Finally we define
\begin{equation}
\label{eq:jump sol explicit}
u(t,\rho) = \begin{dcases}
\widetilde u (t, \rho)  & \rho < S(t)  \\
0 & \rho > S(t).
\end{dcases}
\end{equation}
Solving explicitly is not possible. However, we can prove that
\begin{proposition}
	Let $\alpha \ge 1$, \eqref{eq:jump sol u_0}. Then, the mass of \eqref{eq:jump sol explicit} is a viscosity solution of \eqref{eq:mass equation} and $S(t) \le \overline S(t)$.
\end{proposition}
\begin{proof}
	Since $\alpha \ge 1$, we have that
	\begin{align*}
	\frac{\diff S}{\diff t} &= m(t, S(t)) u(t, S(t)^-)^{\alpha - 1} = m(t, S(t)) \widetilde u(t, S(t))^{\alpha - 1}  \le \alpha M \widetilde u_0 (S_0) ^{\alpha - 1} = \frac{\diff \overline S}{\diff t} .
	\end{align*}
	where $M = \sup m_0$.
	Also $S(0) = \overline S (0)$. Hence the shock is slower than the last characteristic (i.e. $S(t) \le \overline S (t)$) and this implies that there are no outgoing characteristics so the Lax-Oleinik condition is satisfied.
	
	Now, in order to check that it is a viscosity solution, we can repeat the proof of \Cref{thm:vortex is viscosity sol}, replacing \eqref{eq:new 2.4} by \eqref{eq:jump sol RH}.
\end{proof}

\subsection{Two Dirac deltas}
\label{sec:two deltas}

Let us now consider that
\begin{equation}
u_0 = m_1 \delta_{\rho_1} + m_2 \delta_{\rho_2} .
\end{equation}
where $0 \le  \rho_1 < \rho_2$. Then the initial mass $m_0$ is discontinuous, and this creates some technical difficulties.
We will show the viscosity solution for \eqref{eq:mass equation} is the primitive of
\begin{equation}
\label{eq:u two deltas}
u (t, \rho) = \begin{cases}
0 & \rho < \rho_1 \\
\left ( \alpha t \right ) ^{-\frac 1 \alpha} & \rho \in \left[ \rho_1 , S_1 (t) \right]    \\
0 & \rho \in [S_1 (t) , \rho_2 ] \\
\left( \left(  \frac{ \rho - \rho_2}{\alpha m_1 t} \right)  ^{- \frac  \alpha {\alpha - 1}} + \alpha t  \right)^{-\frac 1 \alpha}  & \rho \in [\rho_2, S_2 (t) ] \\
0 & \rho > S_2(t)
\end{cases}
\end{equation}
where
\begin{equation}
\label{eq:u two deltas S_1}
S_1 (t) = \rho_1 + m_1 \left( \alpha t \right) ^{\frac {\alpha - 1} \alpha}  ,
\end{equation}
and
\begin{equation}
\label{eq:two deltas support}
S_2(t) = \rho_2 + \alpha m_1 \mathcal K^{-1} \left(  \frac{m_2}{ \alpha m_1}\right)   t^{\frac 1 \alpha} ,
\end{equation}
with
\begin{equation}
\label{eq:two deltas support scaling}
\mathcal K(\tau) = \int_0^\tau \left( { s }    ^{- \frac  \alpha {\alpha - 1}} + \alpha   \right)^{-\frac 1 \alpha}   \diff s.
\end{equation}
We have the following estimates for the function $\mathcal K^{-1}$: for $\tau \le s_0$ there exists $c(s_0), C(s_0) > 0$ such that
\begin{equation}
\label{eq:two deltas scaling F inv}
	c(s_0)	\tau^{ \frac {\alpha}{\alpha - 1}} \le \mathcal K^{-1} (\tau) \le C(s_0) \tau^{ \frac {\alpha}{\alpha - 1}}.
\end{equation}

This solution is defined for all $t$ such that $S_1(t) \le \rho_2$, i.e. $t \le ( (\rho_2 - \rho_1) / m_1 )^{\frac \alpha {\alpha - 1}} / \alpha.$ For large $t$, $S_1$ would need to be computed from another further Rankine-Hugoniot condition. We will only use the value for $t$ small, so this computation is enough for our purposes.

\begin{figure}[H]
	\centering
	\includegraphics[width = .45\textwidth]{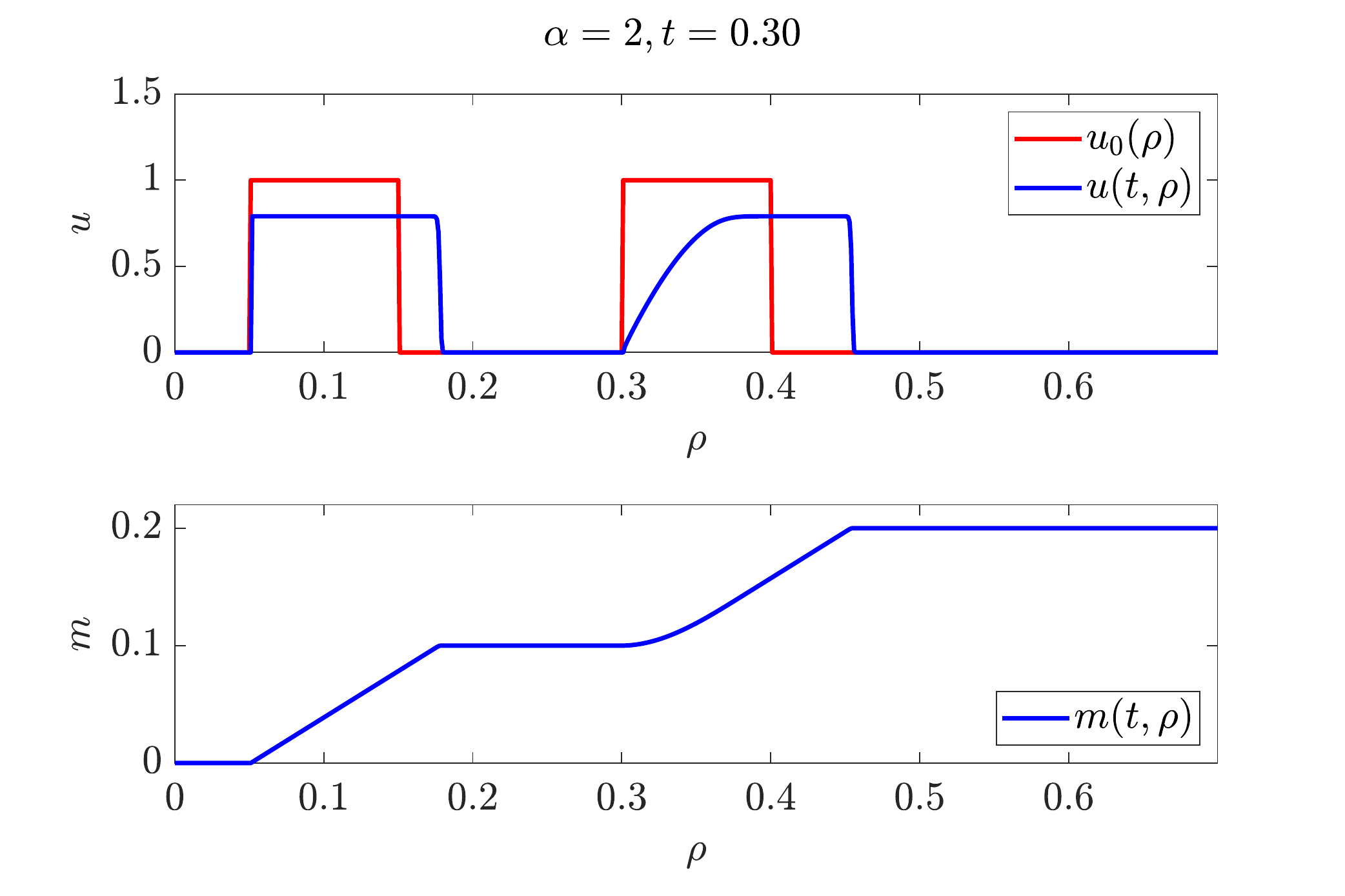}
	\includegraphics[width = .45\textwidth]{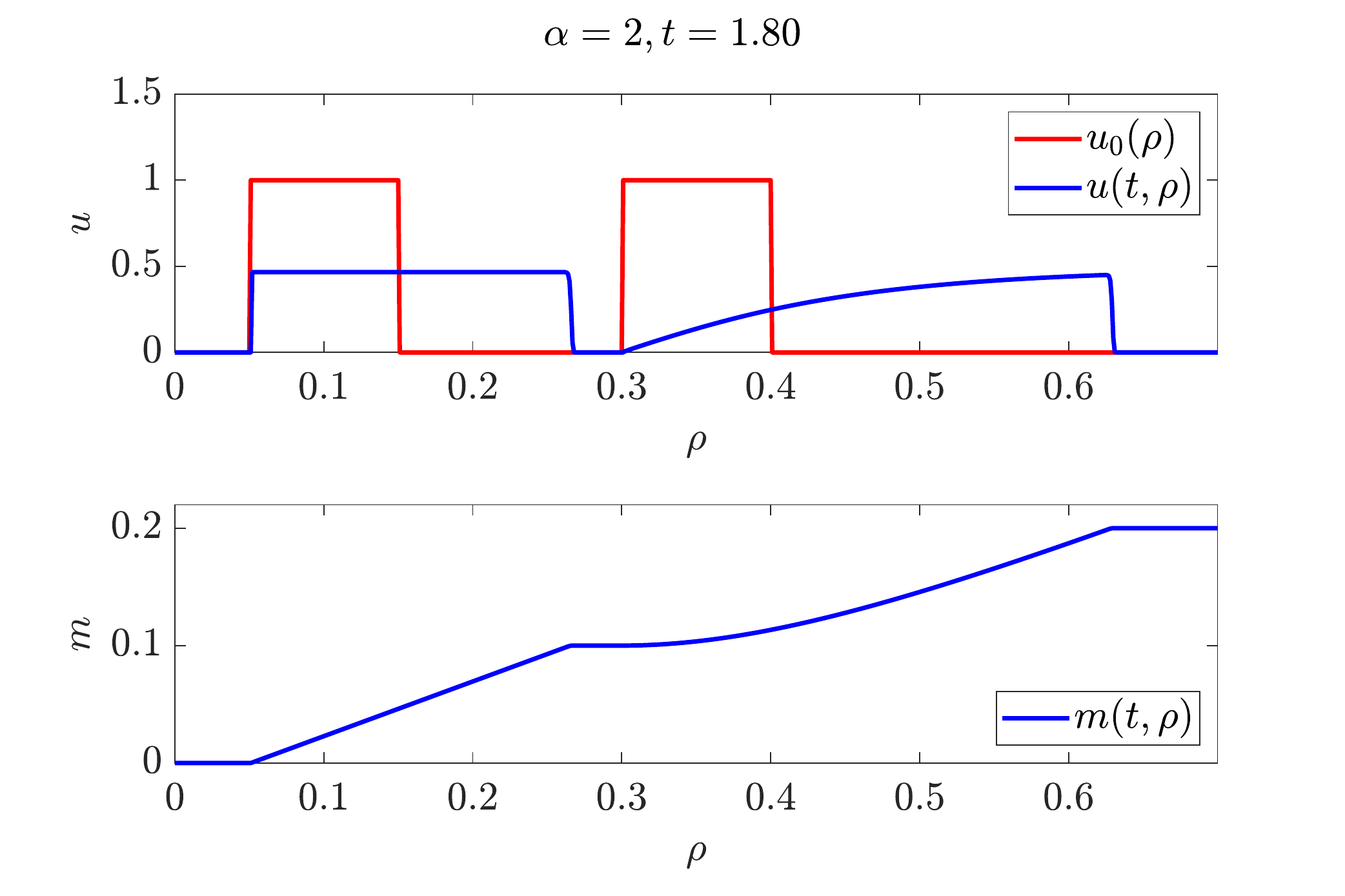}
	\caption{Solutions with $u_0$ given by two characteristics. Computed with the numerical scheme for $m$ in \Cref{sec:numerics} reproducing the exact solution up to approximation error. The function $u = m_\rho$ is recovered by numerical differentiation.}
\end{figure}

\paragraph{Approximation by viscosity solutions}

We will prove there is an explicit solution defined for some $T > 0$ which is of viscosity type for $t > 0$. We will approximate the initial data by
\begin{equation}
	u_0^{\varepsilon,\delta} = \begin{dcases}
	\frac {m_1}  \varepsilon & \rho \in [\rho_1 , \rho_1 + \varepsilon]  \\
	\frac {m_2} {\varepsilon} \frac{(\rho - (\rho_2 - \delta))}{\delta} & \rho \in [\rho_2 - \delta , \rho_2]  \\
	\frac{m_2} {\varepsilon } & \rho \in [\rho_2, \rho_2 + \varepsilon ] \\
	0 & \text{otherwise}
	\end{dcases},
\end{equation}
for $\ee$ and $\delta$ small enough.
The $\ee$-regularisation is used to approximate the Dirac deltas at the level of $u$. The $\delta$-regularisation is used to resolve the appearance of a rarefaction wave at $\rho_2$ due to a gap in the characteristics. Since viscosity solutions are stable by passage to the limit, we only need to show that our approximating solution are viscosity solutions.

The first part of the solutions does not notice the $\delta$-regularisation.
We take $\varepsilon$ small enough so that $\rho_1 + \varepsilon < \rho_2 - 2\varepsilon$.
For $\rho < \rho_2$ we reconstruct a vortex type solution following \Cref{sec:vortex} with an initial gap
\begin{equation}
	u_\varepsilon (t, \rho) = \begin{dcases}
	0 & \rho < \rho_1 \\
	\left (\left( \frac{m_1}{\varepsilon }\right)^{-\alpha} + \alpha t \right ) ^{-\frac 1 \alpha} & \rho \in \left[ \rho_1 , S_1^\varepsilon (t) \right]
	\end{dcases}
\end{equation}
where the first shock is given by
\begin{equation}
	S_1^\varepsilon (t) = \rho_1 + \varepsilon_1 + m_1 \left( \left  (\left( \frac{m_1}{\varepsilon }\right)^{-\alpha} + \alpha t \right) ^{\frac {\alpha - 1} \alpha} - \left( \frac{m_1}{\varepsilon }\right)^{1-\alpha}   \right) .
\end{equation}
Solutions in this form are defined for $t \in [0,T_\varepsilon)$ such that $S_1^\varepsilon(T_\varepsilon) = \rho_2 - \delta $. We leave to the reader to check that $T_\ee$ does not tend to zero with $\ee \to 0$.

For the second part, the characteristics with foot $\rho_0 \in [\rho_2 - \delta  , \rho_2]$ are given by
\begin{equation}
	\label{eq:two deltas approx characteristics 1}
	\rho = \rho_0 + \alpha \left(  m_1 + \frac{m_2}{2} \frac{(\rho_0 + \delta - \rho_2)^2}{\varepsilon \delta }  \right) \left( \frac{m_2}{\varepsilon \delta } (\rho_2 + \delta - \rho)\right)^{\alpha - 1} t
\end{equation}
On the other hand, if $\rho_0 \in [\rho_2 , \rho_2 + \varepsilon]$ we have
\begin{equation}
	\label{eq:two deltas approx characteristics 2}
	\rho = \rho_0 + \alpha \left(  m_1 + \frac{m_2}{2 \varepsilon} \delta + \frac{m_2}{\varepsilon} (  1 - (\rho_2 + \varepsilon - \rho_0) ) \right) \left( \frac{m_2}{\varepsilon}\right)^{\alpha - 1} t.
\end{equation}
Notice that in both cases $u$ is given by
\begin{equation*}
u(t, \rho) = (u_0 (\rho_0)^{-\alpha} + \alpha t  )^{-\frac 1 \alpha}.
\end{equation*}
By mass conservation we have a further shock starting from $\rho_2 + \varepsilon$ given by a Rankine-Hugoniot condition
\begin{equation*}
	\frac{dS_2^{\varepsilon, \delta} }{dt} (t) = (m_1 + m_2) u_{\varepsilon,\delta }(t, S_2^\varepsilon(t)^-)^{\alpha-1}.
\end{equation*}
The first part of solution is of viscosity type, by an argument analogous to \Cref{sec:vortex} and the second part have a monotone non-decreasing datum with final cut-off as in \Cref{sec:monotono non-decreasing with cut-off}. We are reduced now to pass to the limit as $\ee $ and $\delta$ tend to $0$.

\paragraph{Passage to the limit as $\delta \to 0$}
For $[0, \rho_2 - \delta]$ the solution did not depend on $\delta$, so there is no work needed. Applying a similar argument as in \cite{Carrillo+GV+Vazquez2019} we can pass to the limit in \eqref{eq:two deltas approx characteristics 1}.
The characteristics with foot in $[\rho_2, \rho_2 + \delta]$ collapse to a rarefaction fan at $\rho_2$ of the form
\begin{equation}
\label{eq:two delta characteristics collapse}
		\rho = \rho_2 + m_1 \eta_0^{\alpha - 1} t  , \qquad \eta_0 \in \left[ 0,  \frac{m_2}{2\varepsilon} \right ].
\end{equation}
By inverting $\eta_0$ in \eqref{eq:two delta characteristics collapse} we recover the solution
\begin{equation*}
	u_\varepsilon (t, \rho) = \left( \eta_0^{-\alpha} + \alpha t  \right)^{-\frac 1 \alpha} = \left( \left(  \frac{ \rho - \rho_2}{\alpha m_1 t} \right)  ^{- \frac \alpha {\alpha - 1}} + \alpha t  \right)^{-\frac 1 \alpha}.
\end{equation*}

The other characteristics are for foot $\rho_0 \in [\rho_2 , \rho_2 + \varepsilon ]$ and, by passing analogously to the limit in \eqref{eq:two deltas approx characteristics 2}, we have
\begin{equation*}
	\rho = \rho_0 + \alpha \left(  m_1 + \frac{m_2}{\varepsilon} (  1 - (\rho_2 + \varepsilon - \rho_0) ) \right) \left( \frac{m_2}{\varepsilon}\right)^{\alpha - 1} t.
\end{equation*}
Since $u_0^{\varepsilon}$ is non-decreasing the characteristics do not cross.
The Rankine-Hugoniot condition is now
\begin{equation*}
\frac{dS_2^{\varepsilon} }{dt} (t) = (m_1 + m_2) u_{\varepsilon }(t, S_2^\varepsilon(t)^-)^{\alpha-1}.
\end{equation*}

\paragraph{Passage to the limit as $\varepsilon \to 0$}
Passing to the limit we end up only with the rarefaction fan characteristics and recover \eqref{eq:u two deltas}
where the first shock is given by \eqref{eq:u two deltas S_1}
and the second shock, $S_2$ which defines the support, is a solution of the ODE
\begin{equation}
	\begin{dcases}
	\frac{dS_2 }{dt} (t) = (m_1 + m_2) \left( \left(  \frac{ S_2(t)  - \rho_2}{\alpha m_1 t} \right)  ^{- \frac \alpha {\alpha - 1}} + \alpha t  \right)^{-\frac {\alpha-1} \alpha} , \\
	S_2 (0) = \rho_2.
	\end{dcases}
\end{equation}
Notice that this equation is singular at $t = 0$ but it can be rewritten as
\begin{equation*}
		\frac{dS_2 }{dt} (t) = (m_1 + m_2) t^{-\frac{\alpha-1}{\alpha}} \left( t^{\frac 1{\alpha -1}}  \left(  \frac{ S_2(t)  - \rho_2}{\alpha m_1 } \right)  ^{- \frac \alpha {\alpha - 1}} + \alpha   \right)^{-\frac {\alpha-1} \alpha}. \\
\end{equation*}
Since $\frac{\alpha-1}{\alpha} \in (0,1)$  the Cauchy problem is well-posed.
Alternatively, one can write $S_2$ implicitly as the only value such that
\begin{equation*}
	\int_{\rho_2}^{S_2(t)} u(t, \rho) \diff \rho = m_2.
\end{equation*}
In other words,
\begin{equation}
	\label{eq:two deltas S2 implicit}
	\int_{\rho_2}^{S_2(t)} 	\left( \left(  \frac{ \rho - \rho_2}{\alpha m_1 t} \right)  ^{- \frac  \alpha {\alpha - 1}} + \alpha t  \right)^{-\frac 1 \alpha} \diff \rho = m_2.
\end{equation}
This solution is defined for $0 \le  t < T$ where
\begin{equation*}
	T = \frac{1}{\alpha} \left( \frac{\rho_2 - \rho_1 }{m_1} \right)^{\frac{\alpha}{\alpha - 1}}.
\end{equation*}
By scaling analysis on the integral, we can give an algebraic expression of $S_2(t)$. We apply the change of variables $ \rho = \rho_2 + \alpha m_1 s t^{\frac 1 \alpha}$ to deduce
\begin{align*}
m_2 &=  \int_{\rho_2}^{S_2(t)} 	\left( \left(  \frac{ \rho - \rho_2}{\alpha m_1 t} \right)  ^{- \frac  \alpha {\alpha - 1}} + \alpha t  \right)^{-\frac 1 \alpha} \diff \rho = \int_{0}^{ \frac{ S_2(t) - \rho_2}{\alpha m_1 t^{ \frac 1  \alpha} }  } 	\left( \left(  { s t^{- \frac {\alpha-1} {\alpha }} }\right)  ^{- \frac  \alpha {\alpha - 1}} + \alpha t  \right)^{-\frac 1 \alpha} t ^\frac {1}{\alpha} \alpha m_1 \diff s \\
&= \alpha m_1 \int_{0}^{ \frac{ S_2(t) - \rho_2}{\alpha m_1t^{ \frac 1  \alpha} }  } 	\left( { s }    ^{- \frac  \alpha {\alpha - 1}} + \alpha   \right)^{-\frac 1 \alpha}   \diff s
= \alpha m_1 \mathcal K \left(  \frac{S_2(t) - \rho_2}{\alpha m_1 t^{ \frac 1  \alpha} }   \right) .
\end{align*}
Hence, we recover \eqref{eq:two deltas support}.
To show \eqref{eq:two deltas scaling F inv} we simply indicate that, for $s\le s_0$
\begin{equation*}
	 { s }    ^{- \frac  \alpha {\alpha - 1}} \le { s }    ^{- \frac  \alpha {\alpha - 1}} + \alpha  \le C(s_0)  { s }    ^{- \frac  \alpha {\alpha - 1}}
\end{equation*}

\subsection{Waiting time}

\label{sec:waiting time viscous}

\subsubsection{Existence of waiting time}
We turn the explicit solution in  \eqref{eq:Ansatz} into a viscosity subsolution by extending it by zero, that is we define $\underline m (t, \rho)$ as
\begin{equation}
\label{eq:Ansatz viscosity}
\underline m(t, \rho) \defeq \begin{dcases}
0  & \text{if } \rho < c_0 - \alpha ^{\frac 1 \alpha} M (T-t)^{\frac 1 \alpha} \\
\left(  M^{\frac{\alpha}{\alpha - 1}} - \alpha^{\frac{1}{\alpha - 1}} \frac{(c_0 - \rho)^{\frac{\alpha}{\alpha - 1}}}{(T-t)^{\frac{1}{\alpha - 1}}}    \right)^{\frac{\alpha - 1}{\alpha}} 	& \text{if } \rho \in ( c_0 - \alpha ^{\frac 1 \alpha} M (T-t)^{\frac 1 \alpha} , c_0)\\
M & \text{if } \rho \ge  c_0 \\
\end{dcases}.
\end{equation}
\begin{proposition}
	$\underline m (t, \rho)$ is a viscosity subsolution of $m_t + m m_\rho^\alpha = 0$.
\end{proposition}
\begin{proof}
It is clear that $0$ is a solution of $m_t + m m_\rho^\alpha = 0$. So is the second part for $\rho > c_0 - c ^{\frac 1 \alpha} M (T-t)^{\frac 1 \alpha}$, as we have checked in \Cref{sec:waiting time}. At the matching point $\rho =  c_0 - c ^{\frac 1 \alpha} M (T-t)^{\frac 1 \alpha}$ , we have that
\begin{equation*}
	m_\rho(t, \rho^- ) = 0 , \qquad  	m_\rho (t, \rho^+ ) = c \left(  0^+   \right)^{-\frac{1}{\alpha}}  \frac{ (c_0 - \rho)^{\frac 1 \alpha}}{(T-t)^{\frac 1 \alpha}} = + \infty\\
\end{equation*}
This corner does not allow any smooth $\varphi$ to be tangent from above at this point, and hence the condition of viscosity subsolution is trivially satisfied.
\end{proof}

We will denote by $c_0 = \max \supp u_0$, where $u_0 = (m_0)_\rho$, that coincides with the boundary of $m_0 =  M$ in the sense that
\begin{equation}
	\label{eq:c0 support}
	m_0(\rho) < M \, \text{ for } \, \rho < c_0 \quad \text{and} \quad m_0(\rho) = M \, \text{ for } \, \rho \ge c_0.
\end{equation}

\begin{corollary}
	\label{thm:waiting time existence}
	Let $m_0 \in \BUC([0,+\infty))$ and let $c_0 = \max \supp u_0$. If
	\begin{equation}
		\limsup_{\rho \to c_0^-} \frac{M - m_0(\rho)}{(c_0 - \rho)^{\frac{\alpha}{\alpha-1}} }< +\infty,
	\end{equation}
	then there is waiting time as in \Cref{cor:waiting time charact}.
\end{corollary}

\begin{proof}
	First, we prove that
	\begin{equation*}
		\sup_{ \rho \in [0, c_0]} \frac{M - m_0(\rho)}{(c_0 - \rho)^{\frac{\alpha}{\alpha-1}} } < +\infty.
	\end{equation*}
	Let $\rho_k$ be such that
		\begin{equation*}
	 \frac{M - m_0(\rho_k)}{(c_0 - \rho_k)^{\frac{\alpha}{\alpha-1}} } \longrightarrow \sup_{ \rho \in [0, c_0]} \frac{M - m_0(\rho)}{(c_0 - \rho)^{\frac{\alpha}{\alpha-1}} }.
	\end{equation*}
	If the supremum were infinite, since $M - m_0 (\rho_k)$ is bounded, then we have that $\rho_k \to c_0$. This results in
	\begin{equation*}
		\lim_k \frac{M - m_0(\rho_k)}{(c_0 - \rho_k)^{\frac{\alpha}{\alpha-1}} } \le \limsup_{\rho \to c_0^-} \frac{M - m_0(\rho)}{(c_0 - \rho)^{\frac{\alpha}{\alpha-1}} } < + \infty
	\end{equation*}
	leading to a contradiction.
	
	Therefore, there exists $C > 0$ such that for all $\rho \in [0,c_0]$
	\begin{equation*}
		\frac{M - m_0(\rho)}{(c_0 - \rho)^{\frac{\alpha}{\alpha-1}} } \le C.
	\end{equation*}
	In particular, we have that
	\begin{equation*}
				m_0(\rho) \ge M -  C (c_0 - \rho)^{\frac{\alpha}{\alpha-1}}.
	\end{equation*}
	We can apply the convexity of the function $f(x) = x^{\frac \alpha {\alpha - 1}}$ to show that
	\begin{equation*}
		m_0(\rho)^{\frac {\alpha - 1} \alpha} \ge M^{\frac \alpha {\alpha - 1}} - \tfrac{\alpha}{\alpha - 1} M^{\frac 1 {\alpha - 1}} C (c_0 - \rho)^{\frac \alpha {\alpha - 1} }
		=   M^{\frac{\alpha}{\alpha - 1}} - \frac{\alpha^{\frac{1}{\alpha - 1}}} {T^{\frac{1}{\alpha - 1}}}  {(c_0 - \rho)_+^{\frac{\alpha}{\alpha - 1}}}   = \underline m (0, \rho)^{\frac {\alpha - 1} \alpha},
	\end{equation*}
	for a well chosen $T$ (see, e.g. \Cref{fig:comparison}). Therefore, applying the comparison principle \Cref{thm:comparison principle m} then $m \ge \underline m$, and thus $m$ has waiting time.
\end{proof}

\begin{figure}[H]
	\centering
	\includegraphics[width=.32 \textwidth]{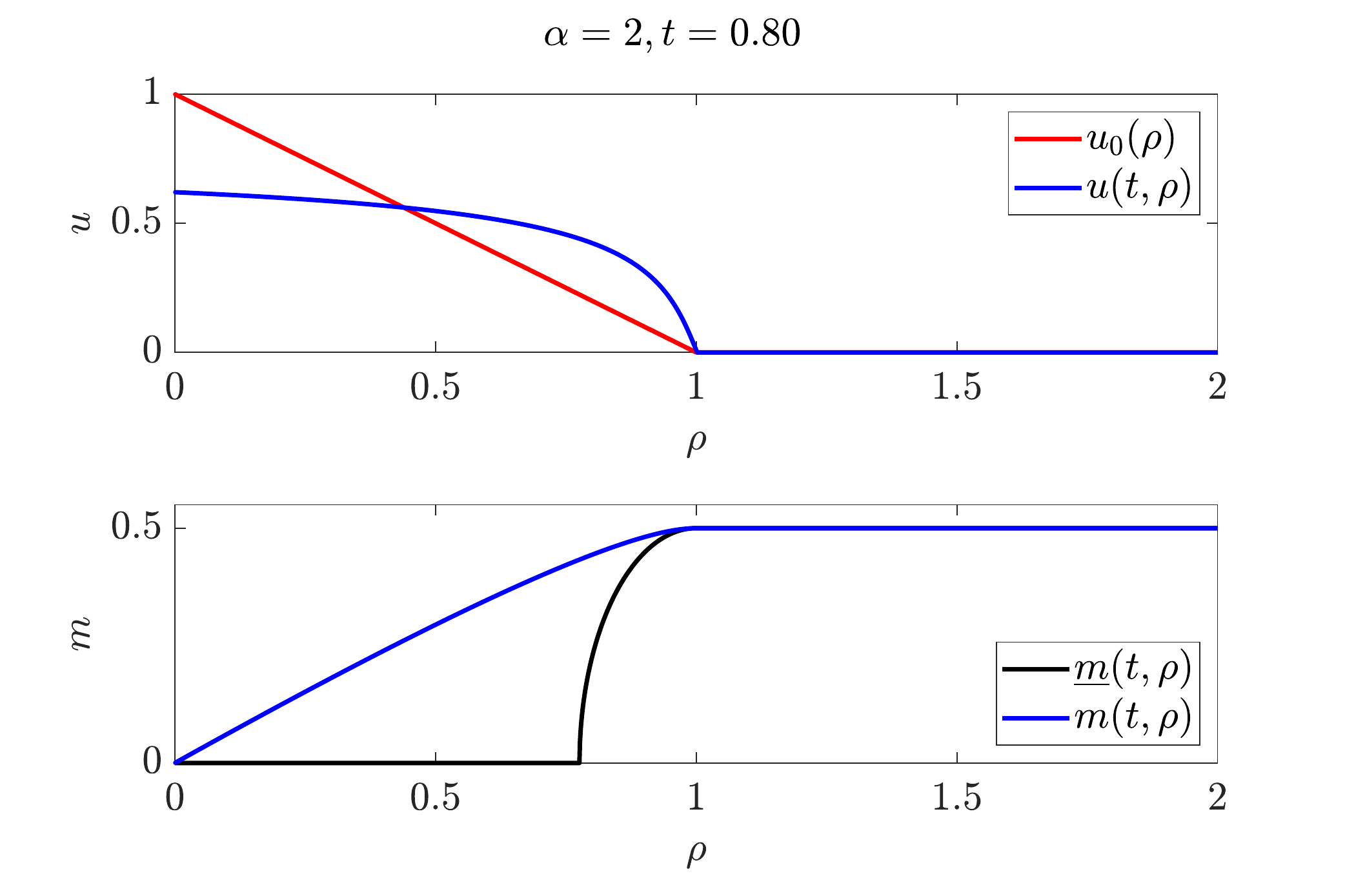}
	\includegraphics[width=.32 \textwidth]{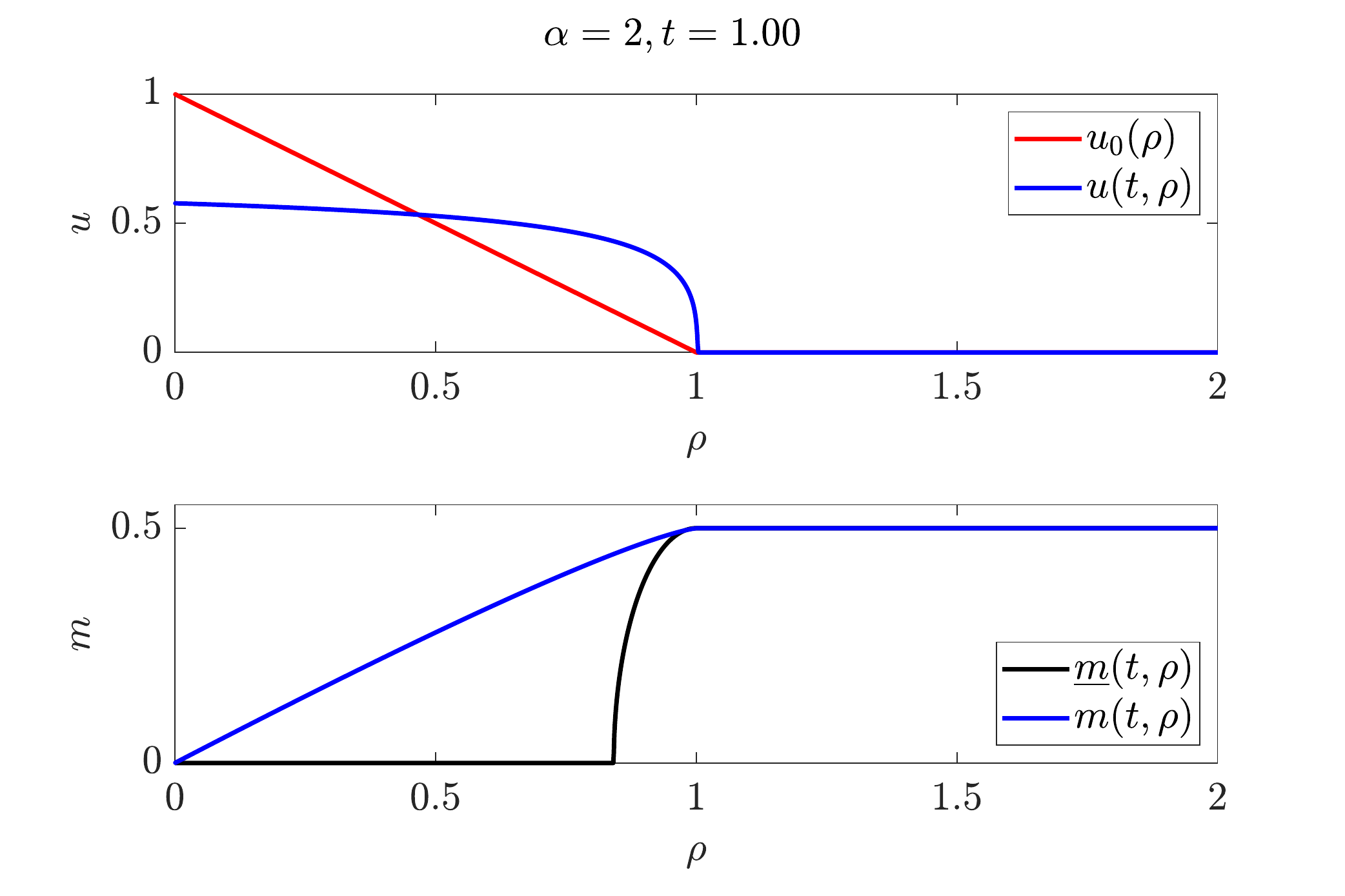}
	\includegraphics[width=.32 \textwidth]{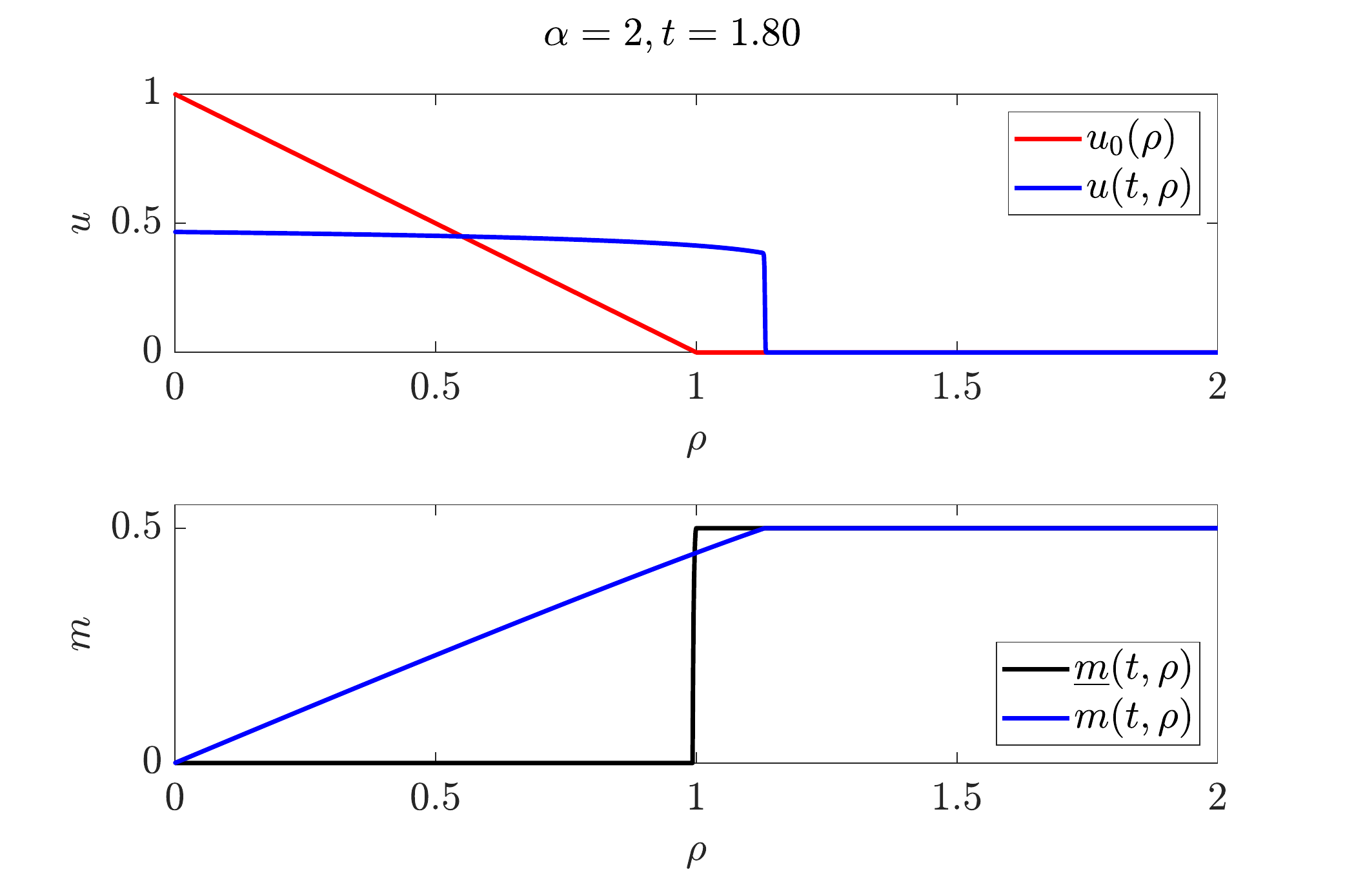}
	\caption{The explicit Ansatz viscosity subsolution \eqref{eq:Ansatz viscosity} guarantees existence of waiting time. The subsolution is represented from the explicit solution, whereas $u$ is computed through the numerical scheme in \Cref{sec:numerics}.
		See a movie simulation in the supplementary material \cite[Video 1]{Videos}.}
	\label{fig:comparison}
\end{figure}

\subsubsection{Non-existence of waiting time}

\begin{theorem}
	\label{thm:waiting time non-existence}
	Let $m_0 \in \BUC ( [0,+\infty)  )$ and let $c_0= \max \supp u_0$. If
	\begin{equation}
		\limsup_{\rho \to c_0^-} \frac{M - m_0(\rho)}{(c_0 - \rho)^{\frac{\alpha}{\alpha-1}} }= +\infty,
	\end{equation}
	then there is no waiting time.
\end{theorem}

\begin{proof}
	To prove there is no waiting time, we want to show that $S(t) > c_0$ for any $t>0$.
	 In order to do this, we will construct a sequence of supersolutions $\overline m_k$ with $S_k (t) = \max \supp (\overline m_k)_\rho (t, \cdot)$ and times $t_k \searrow 0$ such that $S_k (t_k) > c_0$. This ensures that, for some $k$ we have $0 < t_k < t$ and hence $S(t) \ge  S_k (t) \ge S_k (t_k) > c_0$.

	Let us consider a sequence $d_k$ such that
	\begin{equation*}
		d_k \to \limsup_{\rho \to c_0^-} \frac{M - m_0(\rho)}{(c_0 - \rho)^{\frac{\alpha}{\alpha-1}} }.
	\end{equation*}
	There exists $\rho_k \nearrow c_0$ such that
	\begin{equation}
		\label{eq:waiting time comparison}
		M - m_0(\rho_k) \ge d_k ( c_0 - \rho_k )^{\frac{\alpha}{\alpha-1}}.
	\end{equation}
	We construct the viscosity supersolutions $\overline m_k$ with initial derivative
	\begin{equation*}
		u_{k,0} = m_0(\rho_k) \delta_0 + (M - m_0(\rho_k)) \delta_{\rho_k}.
	\end{equation*}
	It is clear that $\overline m_k (0, \rho) \ge m(0, \rho)$. By using the comparison principle \Cref{thm:comparison principle m}, $\overline m_k \ge m$.
	
	 Now we apply the theory of \Cref{sec:two deltas}. We will select $t_k > 0$ such that $S_k(t) \ge  c_0 + \varepsilon_k$ for $t \ge t_k$ with $\varepsilon_k = \frac {c_0 - \rho_k}2 > 0 $. Using \eqref{eq:two deltas support}, \eqref{eq:two deltas support scaling} and \eqref{eq:waiting time comparison} we have that
	\begin{equation*}
			 S_k(t) = \rho_k + \alpha m_0 (\rho_k) \mathcal K^{-1} \left(  \frac{M - m_0 (\rho_k)}{ \alpha m_0(\rho_k) }\right)   t^{\frac 1 \alpha} \ge   \rho_k + \alpha m_0 (\rho_k)  \mathcal K^{-1} \left(  \frac{d_k(c_0 - \rho_k)^{\frac \alpha{\alpha - 1}}}{ \alpha m_0(\rho_k) }\right)   t^{\frac 1 \alpha}.
	\end{equation*}
	Due to our choice of $\rho_k$, it is sufficient that
	\begin{equation*}
		\rho_k + \alpha m_0 (\rho_k)  \mathcal K^{-1} \left(  \frac{d_k(c_0 - \rho_k)^{\frac \alpha{\alpha - 1}}}{ \alpha m_0(\rho_k) }\right)   t^{\frac 1 \alpha}  \ge c_0 + \varepsilon_k  .
		\end{equation*}
	Solving for $t$, we have that
	\begin{equation}
	t \ge   \left( \frac{  \frac{ (c_0 - \rho_k) + \ee_k }{ \alpha m_0 (\rho_k) } }{\mathcal K^{-1} \left(  \frac{d_k(c_0 - \rho_k)^{\frac \alpha{\alpha - 1}}}{ \alpha m_0(\rho_k) }\right)}  \right) ^{\alpha }.
	\end{equation}
	We know that $d_k(c_0 - \rho_k)^{\frac \alpha{\alpha - 1}} \le  M - m_0(\rho_k) \to 0$, therefore we need to study $\mathcal K^{-1}$ close to $0$. Going back to \eqref{eq:two deltas scaling F inv} there exists a constant $C > 0$ such that
	\begin{equation*}
			\left( \frac{  \frac{ (c_0 - \rho_k) + \ee_k }{ \alpha m_0 (\rho_k) } }{\mathcal K^{-1} \left(  \frac{d_k(c_0 - \rho_k)^{\frac \alpha{\alpha - 1}}}{ \alpha m_0(\rho_k) }\right)}  \right) ^{\alpha } \ge C  \left( \frac{  \frac{ (c_0 - \rho_k) + \varepsilon_k }{ \alpha m_0 (\rho_k) } }{\left(  \frac{d_k(c_0 - \rho_k)^{\frac \alpha{\alpha - 1}}}{ \alpha m_0(\rho_k) }\right)^{ \frac {\alpha - 1 } \alpha }}  \right) ^{\alpha }   =  \frac{3^\alpha }{2^{\alpha}} \frac{  C}{ d_k^{  {\alpha - 1 } } } =: t_k .
	\end{equation*}
	Due to the hypothesis of the theorem $d_k \to +\infty$ and hence  $t_k \to 0$.
\end{proof}

See a movie simulation of one of the mass supersolutions interacting with a solution without waiting time in the supplementary material \cite[Video 2]{Videos}.
\begin{remark}
	Notice that if the $\limsup$ is finite, the previous proof can be adapted to show that the supersolutions $\overline m_k$ give an upper bound of the waiting time.
\end{remark}

\begin{remark}
	As pointed out in \Cref{cor:waiting time charact}, the spatial support of classical solutions does not change in time.  Taking $c_0 = \max \supp u_0$, we construct the supersolution $\overline m$ with initial derivative
	\begin{equation*}
	\overline u_{0} = m_0\left( \frac {c_0}2 \right) \delta_0 + \left(M - m_0 \left( \frac {c_0}2 \right) \right) \delta_{\frac {c_0}2 }.
	\end{equation*}
	This supersolution shows that the support of $u$ must move after a finite time and therefore that the solution is no longer classical.
\end{remark}

\subsection{Asymptotic behaviour}
\label{sec:asymptotic behaviour}

We give first a general result of asymptotic behaviour in mass variable.
\begin{theorem}
	Assume that $u_0 \in L^\infty(0,\infty)$ has compact support, $M = \| u_0 \|_{L^1}$, $m$ be the viscosity solution with initial data $m_0$ and $S(t) = \inf \{ \rho : m(t,\rho) = M \}$. Then $S(t) \sim M (\alpha t)^{\frac 1 \alpha}$ with estimate
	\begin{equation}
	\label{eq:S estimate}
		0 \le \frac{S(t)}{M(\alpha t)^{\frac 1 \alpha}} - 1 \le  \frac {S(0)}M (\alpha t)^{-\frac 1 \alpha}.
	\end{equation}
	Furthermore, $m$ has the asymptotic profile in rescaled variable $y = \frac{\rho}{M(\alpha t)^{\frac 1 \alpha}} $ with an asymptotic estimate
	\begin{equation}
		\sup_ { y \ge \varepsilon } \left| \frac{m\left( t ,  M (\alpha t)^{\frac 1 \alpha} y \right) }{M G (y)} - 1\right|  \to 0 , \quad \text{ as } t \to +\infty
		 \quad \text{where }
		G( y   ) = \begin{dcases}
		y  & y\le 1 \\
		 1 & y > 1\\
		 \end{dcases}
	\end{equation}
	for any $\varepsilon > 0$.
\end{theorem}
\begin{proof}
	By \Cref{rem:one Delta data}  we take as super and subsolution $\overline m$ and $\underline m$ with initial data
	\begin{equation*}
		\overline m_0 (\rho) = M H_0 (\rho), \qquad \underline m_0 (\rho) = M H_{S(0)} (\rho).
	\end{equation*}
	Hence $\overline m \ge m \ge \underline m$. Due to the explicit form of $\overline m$ and $\underline m$, we have that
	\begin{equation*}
		M (\alpha t)^{\frac 1\alpha} \le S(t) \le S(0) + M (\alpha t)^{\frac 1\alpha} .
	\end{equation*}
	Due to the self-similar form of $\overline m$ and $\underline m$ given in \Cref{rem:one Delta data}, the result is proven.
\end{proof}

\begin{remark}
	Notice that for $m_0 = 0$ in $[0,a]$, we have $m(t,\rho) = 0$ in $[0, a]$ so the supremum of $y \ge 0$ is always 1. If $u_0$ is continuous and $u_0 (0) > 0$, then the supremum can be taken for $y \ge 0$.
\end{remark}

Let us discuss the asymptotic behaviour when the datum is monotone non-decreasing with final cut-off. We recall \eqref{eq:jump sol u_0}-\eqref{eq:jump sol explicit} and define
\begin{equation*}
	U(t, \xi) = \frac{u(t, (\alpha t)^{\frac 1 \alpha} \xi)}{(\| u_0 \|_{L^\infty}^{-\alpha} + \alpha t )^{-\frac 1 \alpha}}.
\end{equation*}
Since the solution is constructed by characteristics we have that
\begin{equation*}
U(t, \xi) = \begin{dcases}
0 & \text{if } u_0 ((\alpha t)^{-\frac 1 \alpha} \xi) = 0 \\
\left( \frac{ \eta_0 (t,\xi)^{-\alpha} + \alpha t   }{\| u_0 \|_{L^\infty}^{-\alpha} + \alpha t } \right)^{-\frac 1 \alpha} & \text{if } u_0 ((\alpha t)^{-\frac 1 \alpha} \xi) >  0 \text{ and } (\alpha t )^{\frac 1 \alpha} \xi \le S(t) \\
0  & \text{if } (\alpha t )^{\frac 1 \alpha} \xi > S(t)
\end{dcases}
\end{equation*}
where $\eta_0 (t,\xi) \in (0, \| u_0 \|_{L^\infty})$.
Due to \eqref{eq:S estimate}, as $t \to +\infty$ we have that
\begin{equation*}
	U(t, \xi) \to \begin{dcases}
	1 & \text{if } \xi \in (0,M) \\
	0 & \text{if } \xi \in (M,+\infty).
	\end{dcases}
\end{equation*}
The value at $0$ depends on whether $u_0 (0) = 0$.

\section{A numerical scheme}

\label{sec:numerics}

In the pioneering paper by Crandall and Lions \cite{Crandall1984}, the authors developed a theory of monotone schemes for finite differences of Hamilton-Jacobi equations, where solutions are shown to converge to the viscosity solution. They study equations of the form
\begin{equation}
\label{eq:mass general}
m_t + H(m_\rho) = 0.
\end{equation}
For these equations it is natural to develop only explicit methods. However, for our case $m_t + H(m_\rho) m = 0$, we will see that it more natural, and probably more stable, to do an explicit-implicit approximation of the non-linear term $H(m_\rho) m$. In fact, since the nonlinear term is linear in $m$, we can solve for the implicit step in an explicit manner. More precisely, considering an equispaced discretization
\begin{equation}
t_n = h_t n \qquad \rho_j =  h_\rho j.
\end{equation}
We select the following finite-difference schemes
\begin{equation*}
\frac{M_j^{n+1} - M_j^n}{h_t} + \left( \frac{M_{j}^{n} - M_{j-1}^{n}}{h_\rho} \right)^\alpha M_j^{n+1} = 0	
\end{equation*}
which can be written as
\begin{equation}
\label{eq:method interior}
M_j^{n+1} = \frac{ M_j^n  }{1 + h_t \left( \dfrac{M_{j}^n - M_{j-1}^n}{h_\rho} \right)_+^\alpha} = G(M_j^n, M_{j-1}^n).
\end{equation}
Here, $G$ is given by
\begin{equation*}
G(p,q) = \frac{ p  }{1 + h_t H \left( \frac{p-q}{h_\rho} \right)}, \qquad  \text{ where } H(s) = s_+^\alpha.
\end{equation*}
Notice that the method depends only on the parameter $h_t/h_\rho^\alpha$. %
Taking derivatives we have that
\begin{equation*}
\frac{\partial G}{\partial p} = \frac{1 + h_t H \left( \frac{p-q}{h_\rho} \right) - \frac{h_t}{h_\rho} H'\left( \frac{p-q}{h_\rho} \right) p }{\left(  1 + h_t H \left( \frac{p-q}{h_\rho} \right) \right)^2}, \qquad \frac{\partial G}{\partial q} = \frac{ph_tH'\left( \frac{p-q}{h_\rho} \right)}{h_\rho \left(  1 + h_t H \left( \frac{p-q}{h_\rho} \right) \right)^2} \ge 0
\end{equation*}
Then $G$ is non-decreasing in $p$ under the simple CFL condition:
\begin{equation}
\label{eq:CFL general}
\frac{h_t}{h_\rho} H'\left( \frac{p-q}{h_\rho} \right) p \le \frac 1 2.
\end{equation}
Since the denominator in $G$ is larger than $1$, we have that $G(p,q) \le p$. This is immediately translated to a maximum principle for $M_j^n$
\begin{equation}
M_j^{n+1} \le M_j^n \le \| m_0 \|_\infty.
\end{equation}
For $\alpha \ge 1$ we have two options to obtain a CFL condition. We can check whether the numerical derivative is bounded (this can be done for some methods, see \Cref{sec:estimate numerical derivative}) or cut-off the equation by a nice value. For $m_0$ fixed, since $m_\rho \le \| (m_0)_\rho \|_{L^\infty}$ due to \eqref{eq:estimates on solution}, the equation \eqref{eq:mass general}
where $H(s) = s_+^\alpha$ is equivalently to itself with
\begin{equation}
\label{eq:cut-off H}
H(s) =\left(  \max\{ s , \| (m_0)_\rho \|_\infty \} \right)_+^\alpha.
\end{equation}
We write this cut-off to ensure monotonicity. Nevertheless, once the method is monotone, \Cref{lem:boundedness of numerical derivative} ensures that the cut-off part is not reached. Hence, this cut-off is  purely technical.

\medskip

For $\alpha \ge 1$ this new $H$ given by \eqref{eq:cut-off H} satisfies
\begin{equation*}
0 \le H' (s) \le \alpha  \| (m_0)_\rho \|_\infty^{\alpha - 1}.
\end{equation*}
Therefore, \eqref{eq:CFL general} can be taken as
\begin{equation}
\label{eq:CFL}
\tag{CFL}
\frac{h_t}{h_\rho} \le \frac 1 {2 \alpha  \| (m_0)_\rho \|_\infty^{\alpha - 1} \| m_0 \|_\infty }.
\end{equation}
We propose the scheme
\begin{equation}
\tag{M}
\label{eq:method}
\begin{dcases}
{M_j^{n+1}} = \frac{ M_j^n } { 1 +  h_t   H \left(\dfrac {M_j^n - M_{j-1}^n}{h_\rho} \right)  } & \text{if } j>0, n \ge 0 \\
M_j^0 = m_0(h_\rho j) & \text{if } j \ge  0 \\
M_0^n = 0 & \text{if }n > 0.
\end{dcases}
\end{equation}

\begin{remark}
		As we pointed out in \cite{Carrillo+GV+Vazquez2019}, for $0 < \alpha < 1$ this method is not monotone.
		This was fixed by regularising $H$. For $\delta > 0$ we take
	\begin{equation}
	H_\delta( s ) =  ( s_+ + \delta )^\alpha - \delta^\alpha.
	\end{equation}
	By including the boundary and initial condition, we constructed the method
	\begin{equation}
	\tag{M$_\delta$}
	\label{eq:method delta}
	\begin{dcases}
	{M_j^{n+1}} = \frac{ M_j^n } { 1 +  h_t   H_\delta \left(\dfrac {M_j^n - M_{j-1}^n}{h_\rho} \right)  } & \text{if } j>0, n \ge 0 \\
	M_j^0 = m_0(h_\rho j) & \text{if } j \ge  0 \\
	M_0^n = 0 & \text{if }n > 0.
	\end{dcases}
	\end{equation}
	with this regularisation we know that $H_\delta' (s) \le \alpha \delta^{\alpha - 1}$ so we have a CFL condition
	\begin{equation}
	\tag{CFL$_\delta$}
	\label{eq:CFL delta}
	\frac{h_t}{h_\rho} \le  \frac {\delta ^{1-\alpha }} { 2 \alpha \| m_0 \|_\infty  } .
	\end{equation}
	In \cite{Carrillo+GV+Vazquez2019} we made $\delta$ to converge to $0$ with $h_t$ and $h_\rho$, showing the convergence of the numerical solutions.
\end{remark}

\subsection{Properties of monotone methods}
The following properties of \eqref{eq:method} when $G$ is monotone in each variable are a classical matter (see the original result in \cite{Crandall1984} and the presentation and references in \cite{Achdou2013}). We just briefly sketch them for completeness.
\begin{lemma}
	\label{lem:properties of monotone methods}
	Let $\alpha \ge 1$, $m_0 \ge 0$ and bounded and consider the sequence $M_j^n$ constructed by \eqref{eq:method} and assume  \eqref{eq:CFL}. We have the following properties:
	\begin{enumerate}
		\item 	$M_j^{n+1}= G (M_j^n, M_{j-1}^n)$ where
		$G$ is non-decreasing.
		\item $M_j^n \le \|m_0\|_{\infty}$
		\item If $m_0 \ge 0$ is non-decreasing then:
		\begin{enumerate}
			\item $0 \le M_j^n \le M_{j+1}^n$ for all $n, j$
			\item There is mass conservation in the numerical scheme
			\begin{equation*}
			M_\infty^{n+1} = \lim _{j \to +\infty} M^{n+1}_j = \lim_{j \to +\infty} M_j^n = {M_{\infty}^n}.
			\end{equation*}
		\end{enumerate}
	\end{enumerate}
\end{lemma}

\begin{proof}
	\begin{enumerate}
		\item We have shown this above.
		\item This is true for $M_j^0$ by construction, and hence for all $n$, due to the previous item.
		\item
		\begin{enumerate}
			\item We proceed by induction in $n$. For time $n = 0$ this is true due to the monotonicity of $m_0$. Assume $M_j^n \le M_{j+1}^n$ for all $j$. Since $G$ is monotone in each coordinate
			\begin{equation*}
			M_{j+1}^{n+1} = G(M_{j+1}^n, M_j^n) \ge G(M_{j}^n, M_j^n) \ge G (M_{j}^n, M_{j-1}^n) = M_j^{n+1}.
			\end{equation*}
			\item Since the sequence is non-decreasing and bounded, it has a limit. Furthermore $\lim_{j} (M_j^n - M_{j-1}^n) = 0$. Hence, since $H_\delta(0) = 0$ we have that
			\begin{equation*}
			M_\infty^{n+1} = \lim _{j \to +\infty} M^{n+1}_j =  \lim _{j \to +\infty}  \frac{ M_j^n } { 1 +  h_t   H \left(\dfrac {M_j^n - M_{j-1}^n}{h_\rho} \right)  }  = {M_{\infty}^n}. \qedhere
			\end{equation*}
		\end{enumerate}
	\end{enumerate}
	
\end{proof}

	Notice the biggest advantage of the method \eqref{eq:method} is that it preserves the space monotonicity of $m$ and the total mass, as it should be for a mass equation.

\subsection{Convergence of the numerical scheme (\ref{eq:method}) to the viscosity solution}
\label{sec:estimate numerical derivative}

In order to construct a convergent scheme, it is better to work with a single parameter.
For $h > 0$ we define
\begin{equation*}
h_\rho = h, \qquad h_t =  \frac h {2 \alpha  \| (m_0)_\rho \|_\infty^{\alpha - 1} \| m_0 \|_\infty }.
\end{equation*}
so that \eqref{eq:CFL} is satisfied.
We now allow $M_j^n$ to be constructed from \eqref{eq:method}.
For $t_n \le t < t_{n+1}$ and $\rho_j \le \rho < \rho_{j+1}$ we write
the piecewise linear interpolation of the discrete values as
\normalcolor
\begin{equation}
\label{eq:interpolation M}
m^h (t, \rho) = \begin{dcases}
M_j^n + (\rho - \rho_j) \frac{M_{j+1}^n - M_j^n}{h_\rho} + (t - t_n) \frac{M_{j}^{n+1} - M_j^n}{h_t} & \text{if } \rho \le  t \\
M_{j+1}^{n+1} - (\rho_{j+1} - \rho ) \frac{M_{j+1}^{n+1} - M_j^{n+1} }{h_\rho} - ( t_{n+1} - t) \frac{M_{j+1}^{n+1} - M_{j+1}^n}{h_t} & \text{if } \rho > t
\end{dcases}
\end{equation}
This construction ensures that
\begin{equation*}
	\frac{\partial m^h}{\partial \rho} =  \begin{dcases}
 	U_{j+1}^n & \text{if } \rho <  t \\
	U_{j+1}^{n+1}  & \text{if } \rho > t
	\end{dcases}, \qquad \frac{\partial m^h}{\partial t} =  \begin{dcases}
	 - H\left(  U_j^n  \right) M_{j}^{n+1} & \text{if } \rho <  t \\
	 - H\left(  U_{j+1}^n  \right) M_{j+1}^{n+1}  & \text{if } \rho > t
	\end{dcases}
\end{equation*}
where $U_j^n$ is the numerical space derivative
\begin{equation}
\label{eq:U}
U_j^n = \frac{M_j^n - M_{j-1}^n}{h_\rho} \ge 0,
\end{equation}
and the numerical time derivative is given by the relation
\begin{equation*}
\frac{M_{j}^{n+1} - M_j^n}{h_t} = - H\left(  U_j^n  \right) M_{j}^{n+1}  \le 0.
\end{equation*}
The strategy of the proof is the following. We will show that these space and time  numerical derivatives are uniformly bounded, and hence $m^h$ is uniformly continuous, non-decreasing in $\rho$ and non-increasing in $t$. We can then apply the Ascoli-Arzelá precompactness theorem and show there exists a convergent subsequence. We will prove the limit is the viscosity solution.

If we subtract \eqref{eq:method} for $j$ and $j-1$ we recover an equation for the numerical derivative $U_j^n$
\begin{equation}
\label{eq:U equation}
\frac{ U_j^{n+1} - U_j^n}{h_t} + \frac{H(U_j^{n})M_j^{n+1} - H( U_{j-1}^{n} )M_{j-1}^{n+1}  }{h_\rho} = 0.
\end{equation}
Notice that the natural scaling for this equation is $h_t / h_\rho$.

\subsubsection{Boundedness of the numerical derivative}
Since \eqref{eq:U equation} is a numerical approximation by a monotone method of the nonlinear conservation law \eqref{eq:main equation}, we can expect a maximum principle.
\begin{lemma}
	\label{lem:boundedness of numerical derivative}
	Let $0 \le m_0$ be uniformly Lipschitz continuous, bounded and non-decreasing, $M_j^n$ be given by \eqref{eq:method}, that \eqref{eq:CFL} holds and let $U_j^n$ given by \eqref{eq:U}. Then, $U_j^n \ge 0$ and
	\begin{equation}
	\sup_j U_j^{n+1} \le \sup_j U_j^n \qquad \forall n \ge 0.
	\end{equation}
\end{lemma}

\begin{remark}
	Once this is proven, the cut-off \eqref{eq:cut-off H} is not needed.
\end{remark}

\begin{proof} That $U_j^n \ge 0$ follows form \Cref{lem:properties of monotone methods}. We write
	\begin{align*}
	0&=\frac{ U_j^{n+1} - U_j^n}{h_t} + M_{j-1}^{n+1} \frac{H(U_j^{n}) - H( U_{j-1}^{n} )  }{h_\rho} +  H( U_{j}^{n} ) \frac{M_{j}^n - M_{j-1}^n}{h_\rho} \\
	&=\frac{ U_j^{n+1} - U_j^n}{h_t} + M_{j-1}^{n+1} \frac{H(U_j^{n}) - H( U_{j-1}^{n} )  }{h_\rho} +  H( U_{j}^{n} ) U_j^n.
	\end{align*}
	Solving for $U_j^{n+1}$, using the fact that $H$ is non-decreasing and $U_j^n \ge 0$, we have that
	\begin{align*}
	U_j^{n+1} &\le U_j^n -	\frac{h_t}{h_\rho} M_{j-1}^{n+1} (H(U_j^{n}) - H( U_{j-1}^{n} )  ) \\
	& = U_j^n -\frac{ h_t }{h_\rho}  M_j^{n+1}  H'(\xi_j^n ) U_j^n + \frac{ h_t }{h_\rho}  M_j^{n+1}   H' (\xi_j^n )  U_{j-1}^{n} \\
	&= \left( 1 -  \frac{ h_t }{h_\rho}  M_j^{n+1}  H'(\xi_j^n ) \right) U_j^n  + \frac{ h_t }{h_\rho}  M_j^{n+1}   H' (\xi_j^n )  U_{j-1}^{n} .
	\end{align*}
	Due to \eqref{eq:CFL} we have that the coefficients in front of $U_j^n$ and $U^{n}_{j-1}$ are non-negative. Hence
	\begin{align*}
	U_j^{n+1} &\le \left( 1 -  \frac{ h_t }{h_\rho}  M_j^{n+1}  H'(\xi_j^n ) \right) \sup_j U_j^n  + \frac{ h_t }{h_\rho}  M_j^{n+1}   H' (\xi_j^n )  \sup_j U_j^n  \\
	&=   \sup_j U_j^n .
	\end{align*}
	And this holds for every $j$ so the result is proved.
\end{proof}

\subsubsection{Convergence via Ascoli-Arzelá. Existence of a viscosity solution}

\begin{theorem}
	\label{thm:alpha > 1 existence of m}
	Let us $m_0 \in W^{1,\infty} (0,+\infty)$ and non-decreasing, \eqref{eq:CFL}, $M_j^n$ constructed by \eqref{eq:method} and $m^h$ be given by \eqref{eq:interpolation M}. Then, $m^{h}$ is a family of uniformly continuous functions. Then, for every $P > 0$
	\begin{equation}
	m^{h} \to m \qquad \text{ in } \mathcal C ([0,P] \times [0,T]) \text{ as  } h_\rho \to 0.
	\end{equation}
	where $m$ is a viscosity solution of \eqref{eq:mass equation}. Furthermore, \eqref{eq:estimates on solution} holds.
\end{theorem}

\begin{proof}
	First, we notice that $m^h$ satisfies \eqref{eq:estimates on solution}. Due to \eqref{eq:interpolation M}, we have that
	\begin{equation*}
	|m_\rho^h (t,\rho)| \le \| (m_0)_\rho \|_\infty, \qquad |m_t^h (t,\rho)| \le |H(U_j^n) M_j^{n+1}| \le  H( \| (m_0)_\rho \|_\infty) \| m_0 \|_\infty.
	\end{equation*}
	By the Ascoli-Arzelá theorem there is a subsequence that converges uniformly in $[0,P] \times[0,T]$.

	We will show every convergent subsequence converges to the same limit $m$, and hence the whole sequence converges. We still denote by $h$ the indices of the convergent subsequences.
	
	Let $m^h$ be a subsequence converging in $[0, T] \times [0, P]$.
	We check that it is a viscosity subsolution, and the proof of viscosity supersolution is analogous. Let
	$(t_0, \rho_0) \in (0, T) \times (0, P)$ and
	$\varphi \in C^2$ be
	 such that $m - \varphi$ has a strict local maximum at $(t_0,\rho_0)$ and $m(t_0, \rho_0) = \varphi(t_0, \rho_0)$.
	We can modify $\varphi$ outside a bounded neighbourhood of $(t_0, \rho_0)$, so that $m -\varphi$ attains a unique global maximum at $(t_0, \rho_0)$,
	for $h$ large enough $m^h - \varphi$ attains a global maximums in $[0, T] \times [0, P]$ at an interior points $(t_h, \rho_h)$, and $(t_h, \rho_h) \to (t_0, \rho_0)$ as $h \to 0$. Our argument is a variation of \cite[Lemma 1.8]{Tran2019}.
	
	 Let $B \subset [0, T] \times [0, P]$ be a small open ball around $(t_0, \rho_0)$ where the maximum is global. Let $\ee = - \inf_B (m - \varphi) / 2$. Define $U = \{ m - \varphi > - \ee \} \cap B$ which is a open and bounded neighbourhood of $(t_0, \rho_0)$. We modify $\varphi$ so that is greater than $m + \varepsilon$ outside $U$. With the modification, $m-\varphi$ attains a unique global maximum at $(t_0, \rho_0)$.
	
	Let $h$ be small enough so that $|m^h - m| < \frac {\ee}{2}$ in $[0, T] \times [0, P]$. We have that
	\begin{equation*}
		\max_{[0,t_0 + 1] \times [0,\rho_0+1] \setminus U } (m^h - \varphi) < \max_{[0,t_0 + 1] \times [0,\rho_0+1] \setminus U } (m - \varphi) + \frac \ee 2 \le  - \frac \ee 2.
	\end{equation*}
	On the other hand
	$$m^h(t_0, \rho_0) - \varphi(t_0, \rho_0) >   m(t_0, \rho_0) - \varphi(t_0, \rho_0) - \frac {\ee } {2} = - \frac{\ee}{2}.$$
	Therefore the maximum over $[0, T] \times [0, P]$ is attained at some $(t_h, \rho_h) \in U$. The sequence $(t_h, \rho_h)$ is bounded, and therefore as a convergent subsequence. Let $(t_1, h_1)$ be its limit. We have that
	\begin{equation*}
		m^h(t_h, \rho_h) - \varphi(t_h, \rho_h) \ge m^h(t, \rho) - \varphi(t, \rho) \qquad  \forall (t, \rho) \in [0, T] \times [0, P].
	\end{equation*}
	Passing to the limit, since the maximum is unique, we have that $(t_1, \rho_1) = (t_0, \rho_0)$. Since all convergent subsequences share a limit, the whole sequence converges.
	
	For such small values of $h$,
	let us define
	\begin{equation*}
	{n_h} =  \left \lfloor  \frac{t_h}{h_t} \right \rfloor -1 , \qquad {j_h} = \left \lceil  \frac{\rho_h}{h_\rho} \right \rceil .
	\end{equation*}
	Since $m^h - \varphi $ has a global maximum in $[0,T] \times [0,P]$, we have that
	\begin{equation*}
	m^h (t_h , \rho_h) - \varphi(t_h, \rho_h) \ge m^h (t , \rho) - \varphi(t , \rho ).
	\end{equation*}
	Evaluating on the points of the mesh, we get that
	\begin{equation*}
	M_j^n \le \varphi (t_n, \rho_j) - \varphi (t_h , \rho_h) + m^h (t_h , \rho_h).
	\end{equation*}
	Since $m^h$ is increasing in $\rho$ and decreasing in $t$ and the fact that $G$ is non-decreasing, we recover
	\begin{align*}
	m^h (t_h , \rho_h) & \le  m^h((n_h+1) h_t, j_h h_\rho ) = M_{j_h  }^{n_h + 1}     =  G( M_{j_h }^{n_h},M_{j_h - 1}^{n_h} )   \\
	& \le G\Big ( \varphi (t_{n_h}, \rho_{j_h}) -  \varphi (t_h , \rho_h) + m^h (t_h , \rho_h) ,    \varphi (t_{n_h}, \rho_{j_h - 1}) -  \varphi (t_h , \rho_h) + m^h (t_h , \rho_h) \Big  ) \\
	&= \frac{\varphi (t_{n_h}, \rho_{j_h}) -  \varphi (t_h , \rho_h) + m^h (t_h , \rho_h) }
	{1 + h_t H \left(  \frac{\varphi (t_{n_h}, \rho_{j_h}) - \varphi (t_{n_h}, \rho_{j_h - 1})} {h_\rho}     \right) }
	\end{align*}
	due to the definition of $G$. Since $\varphi$ is smooth, for $h$ small enough the denominator is positive and hence
	\begin{align*}
	\frac{\varphi (t_h , \rho_h) - \varphi (t_{n_h}, \rho_{j_h})  }{h_t } + H\left(  \frac{ \varphi (t_{n_h}, \rho_{j_h})  -  \varphi (t_{n_h}, \rho_{j_h-1})  }{h_\rho }   \right) m^h (t_h, \rho_h) \le 0 .
	\end{align*}
	Adding and subtracting $\varphi (t_{n_h + 1}, \rho_{j} ) / h_\rho $ on both sides
	\begin{align*}
	\frac{\varphi (t_{n_h+1} , \rho_{j_h}) - \varphi (t_{n_h}, \rho_{j_h})  }{h_t } &+ H\left(  \frac{ \varphi (t_{n_h}, \rho_{j_h})  -  \varphi (t_{n_h}, \rho_{j_h-1})  }{h_\rho }   \right) m^h (t_h, \rho_h) \\
	&\le  \frac{ \varphi (t_{n_h+1} , \rho_{j_h}) -   \varphi (t_h , \rho_h) }{t_h - t_{n_h}} \frac {t_h - t_{n_h}} {h_t}  .
	\end{align*}
	Clearly $t_h - t_{n_h} \ge 0$ and, since $\varphi$ is of class $C^1$ and $m$ is non-increasing in $t$, we have that
	\begin{equation*}
	\lim_{h \to 0} \frac{ \varphi (t_{n_h+1} , \rho_{j_h}) -   \varphi (t_h , \rho_h) }{t_h - t_{n_h}} = \varphi_t (t_0, \rho_0) \le 0,
	\end{equation*}
	Therefore, as $h\to 0$, we conclude that
	\begin{equation*}
	\varphi_t (t_0, \rho_0) + H\left(  \varphi_\rho (t_0,\rho_0)   \right) m (t_0, \rho_0) \\
	\le 0 ,
	\end{equation*}
	for any $\varphi$ such that $m - \varphi $ has a global maximum at $(t_0, \rho_0)$. Therefore, $m$ is a viscosity subsolution.
\end{proof}

\subsection{Rate of convergence}

\begin{theorem}
	\label{thm:convergence M to m}
	Let $\alpha \ge 1$ and let $h_t$ and $h_\rho$ satisfy \eqref{eq:CFL}. Let $m_0$ be Lipschitz continuous and bounded and let $m$ be the viscosity solution of \eqref{eq:mass equation} and $M_j^n$ be constructed by \eqref{eq:method}. Then, for any $T > 0$
	\begin{equation*}
	\sup_{ \substack{ j \ge 0 \\  0 \le n \le T/h_t} } | m  (t_n,\rho_j) - M_j^n | \le C h_\rho^{\frac {1}{3}}.
	\end{equation*}	
	where $C$ does not depend on $h_\rho$.
\end{theorem}
\begin{remark}
	The original paper by Crandall and Lions allows, by a longer and more involved argument, proves estimates of the form $O(\sqrt{h_t})$ with $H$ continuous, but requiring that the function defining the method is Lipschitz continuous.
\end{remark}

\begin{proof}
	For convenience, in the proof we denote
	$
	N = \lceil T / h_t \rceil .
	$
	Our aim is to prove that
	\begin{equation*}
	\sigma = \sup_{ \substack{ j \ge 0 \\  0 \le n \le T/h_t} } ( m (t_n,\rho_j) - M_j^n ) \le C h_\rho^{\frac 1 3}.
	\end{equation*}
	The argument can be analogously repeated for the infimum.
	If $\sigma \le 0$ there is nothing to prove. Assume that $\sigma > 0$. Let $L$ be the Lipschitz constant of $m_0$. Due to \eqref{eq:estimates on solution}, it is also the Lipschitz constant of $m$.
	
	\medskip
	
	We begin by indicating there exist $n_1,j_1$ such that
	\begin{equation*}
	m (t_1, \rho_1) - M_{j_1}^{n_1} \ge \frac {3\sigma} 4, \qquad \text{ where } t_1 = h_t {n_1} \text{ and } \rho = h_\rho {j_1}.
	\end{equation*}
	We define
	\begin{align*}
	\Phi(t,h_t n,\rho,h_\rho j) &= m(t,\rho) - M_j^n - \phi(t,h_t n,\rho,h_\rho j)
	\end{align*}
	where, for $\ee, \lambda > 0$ we define
	\begin{equation*}
	\phi(t,s,\rho,\xi) = \left( \frac{|\rho - \xi|^2 + |t-s|^2}{\ee^2}  + \lambda (t + s) \right)
	\end{equation*}
	Then the maximum is at $t_\ee \in [0,T] , \rho_\ee \in [0,+\infty)$ and $t_\ee = h_t n_\ee$ with $n_\ee \in \{0, \cdots, N\} , \xi_\ee = h_\rho j_\ee$ with $j_\ee \in \mathbb N \cup \{0\}$. 	Again this function is continuous and
	\begin{enumerate}
		\item Defined over a bounded set in $t_\ee$ and $s_\ee$.
		\item If $\rho \to +\infty$ and $j$ remains bounded then $\Phi \to -\infty$ (analogously in $\rho$ bounded and $j \to + \infty$).
		\item If $\rho \to + \infty$ and $j \to + \infty$ then
		$$	\limsup_{\rho,j\to + \infty} \Phi \le  m_\infty - m_\infty  \le 0, $$
	\end{enumerate}
	so there exists a point of maximum $(t_\ee,h_tn_\ee,\rho_\ee,h_\rho j_\ee)$ such that
	\begin{equation*}
	\Phi(t_\ee , s_\ee , \rho_\ee ,\xi_\ee) \ge \Phi(t,s,\rho, \xi) \qquad \forall (t,s,\rho,\xi).
	\end{equation*}
	In particular
	\begin{equation*}
	\Phi(t_\ee , s_\ee , \rho_\ee ,\xi_\ee) \ge \Phi(t_1,t_1,\rho_1, \rho_1) = m(t_{1},\rho_{1}) - M_{j_1}^{n_1}  - 2 \lambda t_1 .
	\end{equation*}
	Taking
	\begin{equation}
	\label{eq:condicion ee 1}
	\lambda = \frac{ \sigma }{ 8  (1 + T ) }
	\end{equation}
	we have
	\begin{equation*}
	\Phi(t_\ee , \rho_\ee,n_\ee, j_\ee) \ge \frac \sigma 2.
	\end{equation*}
	In particular,
	\begin{equation}
	\label{eq:estimate maximum}
		m(t_\ee,\rho_\ee) - M_{j_\ee}^{n_\ee} \ge \frac \sigma 2  + \phi(t_\ee ,h_t n_\ee,\rho_\ee,h_\rho j_\ee) > 0
	\end{equation}
	\noindent \textbf{Step 1. Variables collapse.} As $\Phi(t_\ee , s_\ee , \rho_\ee ,\xi_\ee) \ge \Phi(0,0,0,0) = 0$, we have
	\begin{equation*}
	\frac{|\rho_\ee-\xi_\ee|^2 + |t_\ee-s_\ee|^2}{\varepsilon^2} + \lambda (t_\ee + s_\ee) \le m(t_\ee, \rho_\ee) - M_{j_\ee}^{n_\ee} \le 2\| m_0 \|_\infty .
	\end{equation*}
	Therefore, we obtain
	\begin{equation*}
	|\rho_\ee - \xi_\ee|^2 + |t_\ee - s_\ee|^2 \le 2\| m_0 \|_\infty  \ee^2, \qquad \text{and} \qquad \rho_\ee^2 + \xi_\ee^2 \le \frac{{ 2\| m_0 \|_\infty } }{ \ee}.
	\end{equation*}
	This implies that, as $\varepsilon \to 0$, the variable doubling collapses to a single point.
	
	\noindent \textbf{Step 2. For $\varepsilon$ small enough, $t_\ee,\rho_\ee,n_\ee, j_\ee > 0$.} Since $m$ is Lipschitz continuous
	\begin{align*}
	\frac{\sigma}{2} &<  m(t_\ee, \rho_\ee) - M_{j_\ee}^{n_\ee} \\
	&=m(t_\ee, \rho_\ee) - m(0, \rho_\ee) + m(0, \rho_\ee) - m(0, \xi_\ee)\\
	&\qquad  + m(0, \xi_\ee)  - M_{j_\ee}^0  + M_{j_\ee}^0 - M_{j_\ee}^{n_\ee} \\
	&\le L t_\ee  + L |\rho_\ee - \xi_\ee| ,
	\end{align*}
	using the fact that $m(0,\xi_\ee) = M_{j_\ee}^0$ and $M_j^n$ is decreasing in $n$.
	If $\ee$ is small enough
	\begin{equation}  \label{eq:condicion ee 2}
	\ee < \frac{ \sigma } {4L \sqrt{ 2\| m_0 \|_\infty }},	
	\end{equation}
	we have $L t_\ee >  \sigma / 4$ and hence $t_\ee > 0$. Analogously for $\rho_\ee > 0$.
	
	If $n_\ee = 0$ then
	\begin{align*}
	\frac{\sigma}{2} &<  m(t_\ee, \rho_\ee) - M_{j_\ee}^{0} \\
	&=m(t_\ee, \rho_\ee) - m(0, \rho_\ee) + m(0,\rho_\ee) - m(0,\xi_\ee)  + m(0, \xi_\ee)  - M_{j_\ee}^0  \\
	&\le L t_\ee  + L |\rho_\ee - \xi_\ee| = L |t_\ee - n_\ee| + L |\rho_\ee - \xi_\ee| \\
	&\le L \sqrt{ 2 \| m _0 \|_\infty }\ee .
	\end{align*}
	This is a contradiction if \eqref{eq:condicion ee 2} holds. An analogous contradiction holds if $j_\ee = 0$.
	
	\medskip

	\noindent \textbf{Step 3. An inequality for $m$ via viscosity.} We check that
	\begin{equation*}
	(t,\rho) \longmapsto m(t, \rho) - \phi(t,s_\ee,\rho,\xi_\ee) = m(t, \rho) - \psi(t, \rho)
	\end{equation*}
	has a maximum at $(t_{\ee}, \rho_{\ee})$. Hence, since $m$ is a viscosity subsolution, we have
	\begin{equation*}
	\frac{\partial \phi}{\partial t} (t_\ee,s_\ee, \rho_\ee, \xi_\ee) + H \left( \frac{\partial \phi}{\partial \rho} (t_\ee,s_\ee, \rho_\ee, \xi_\ee) \right) m (t_\ee, \rho_\ee) \le 0.
	\end{equation*}
	Computing the derivatives
	\begin{equation}
	\label{eq:FD viscosity m}
	\frac{2(t_\ee -s_\ee)}{\ee^2} + \lambda  + H \left(  \frac{2 (\rho_\ee  - \xi_\ee) }{\ee^2}     \right) m (t_\ee, \rho_\ee) \le 0.
	\end{equation}

	\noindent \textbf{Step 4. An inequality for $M$ applying that $G$ is monotone.}  As before, the function
	\begin{equation*}
	(n,j) \longmapsto M_j^n - \left(- \phi(t_\ee,h_t n,\rho_\ee,h_\rho j) \right) = M_j^n - \psi(j,n)
	\end{equation*}
	has a minimum at $(n_\ee, j_\ee)$. In particular, we obtain
	\begin{equation*}
	M_j^n  \ge M_{j_\ee}^{n_\ee} - \psi(j_\ee, n_\ee)  + \psi(j, n).
	\end{equation*}
	Since $G$ is monotone, it yields
	\begin{align*}
	M_{j_\ee}^{n_\ee} &= G(M_{j_\ee}^{n_\ee-1},M_{j_\ee-1}^{n_\ee-1}) \\
	&\ge G\Bigg( \underbrace{
		M_{j_\ee}^{n_\ee} - \psi(j_\ee, n_\ee) + \psi(j_\ee, n_\ee -1 ) }_{S_1}, \underbrace{ M_{j_\ee}^{n_\ee} - \psi(j_\ee, n_\ee) + \psi(j_\ee - 1 , n_\ee -1 ) }_{S_2} \Bigg).
	\end{align*}
	Similarly to the proof of \Cref{thm:alpha > 1 existence of m}, for $h$ small, one can rewrite the previous inequality as
	\begin{equation*}
	\frac{M_{j_\ee}^{n_\ee} - S_1 } {h_t}  + H \left( \frac{S_1 - S_2}{h_x}  \right) M_{j_\ee}^{n_\ee} \ge  0.
	\end{equation*}
	Hence, we recover
	\begin{equation}
	\frac{\psi(j_\ee, n_\ee) - \psi(j_\ee, n_\ee -1 ) } {h_t}  + H \left( \frac{  \psi(j_\ee, n_\ee -1 ) - \psi(j_\ee - 1 , n_\ee -1 )}{h_x}  \right)  M_{j_\ee}^{n_\ee} \ge 0.
	\end{equation}
	We could compute this explicitly, but it is sufficient and clearer to apply the intermediate value theorem to deduce
	\begin{equation*}
	-\frac{\partial \phi}{\partial s} (t_\ee,\bar s_\ee, \rho_\ee,  \xi_\ee) + H \left( - \frac{\partial \phi}{\partial \xi} (t_\ee, { s}_\ee - h_t, \rho_\ee,  {\bar \xi}_\ee) \right) M_{j_\ee}^{n_\ee} \ge 0
	\end{equation*}
	where $\bar s_\ee \in (s_\ee - h_t, s_\ee)$ and $\bar \xi_\ee \in (\xi_\ee - h_\rho, \xi_\ee)$.
	Hence, we conclude that
	\begin{equation}
	\label{eq:FD viscosity m 2}
	\frac{2 (t_\ee - \bar s_\ee )}{\ee^2} - \lambda + H \left( \frac{2 (\rho_\ee-\bar \xi_\ee)}{ \ee^2}  \right) M_{j_\ee}^{n_\ee} \ge 0.
	\end{equation}
	
	\noindent \textbf{Step 5. An estimate for $\sigma$.} Substracting \eqref{eq:FD viscosity m} from \eqref{eq:FD viscosity m 2} we have that
	\begin{align*}
	\frac{ \sigma }{ 4 ( 1 + T) } &\le \frac{s_\ee - \bar s_\ee}{\ee^2}
	+ H \left( \frac{2 (   \rho_\ee - \bar \xi_\ee) }{ \ee^2}    \right) M_{j_\ee}^{n_\ee} - H  \left(  \frac{2 (\rho_\ee  - \xi_\ee) }{\ee^2}     \right) m (t_\ee, \rho_\ee)  \\
	& \le \left(  H \left( \frac{2 (   \rho_\ee - \bar \xi_\ee) }{ \ee^2}      \right) - H  \left(  \frac{2 (\rho_\ee  - \xi_\ee) }{\ee^2}     \right)  \right)M_{j_\ee}^{n_\ee} +  H \left( \frac{2 (   \rho_\ee - \bar \xi_\ee) }{ \ee^2}      \right) ( M_{j_\ee}^{n_\ee} - m(t_\ee, \rho_\ee))
	\end{align*}
	Notice that the second term is non-positive due to \eqref{eq:estimate maximum}. We now use the Lipschitz continuity of $H$, which holds for the cut-off given by \eqref{eq:cut-off H}, and we obtain that
	\begin{align*}
	\frac{ \sigma }{ 8 } &\le C \left| \frac{2 (\bar \xi_\ee - \xi_\ee)}{\ee^2}  \right| \le  {C} \frac{h_\rho}{\ee^{2}} \| m_0 \|_{L^\infty}.
	\end{align*}

	\noindent \textbf{Step 6. A first choice of $\ee$.} We take $\ee  =C \sigma$ 	where $C$ is chosen so that \eqref{eq:condicion ee 2} hold. Then, we have that
	\begin{equation*}
	\sigma^{3}  \le C h_\rho  .
	\end{equation*}
	where $C (T, \alpha)$ is independent of $h_t$, $h_x$ or $\delta$. This completes the proof.
\end{proof}
\begin{remark}
	Notice that we do not use the equation until step 4 and that the Lipschitz continuity of $m_0$ plays a key role. However, the homogeneous boundary conditions do not.	
\end{remark}
\begin{remark}
	Notice that we recover the exponent $h^{\frac 1 3}$ from the Lipschitz continuity of $H$. If $H$ is only $\alpha$-Hölder continuous as in \cite{Carrillo+GV+Vazquez2019}, then the rate of convergence is given by $h^{\frac \alpha { 1+ 2\alpha }}$.
\end{remark}

\section{Numerical results}

\subsection{Asymptotics as $t \to +\infty$}
Through numerical experiments, we see that the vortex seems to be the asymptotic solution also in $u$ variable. In \Cref{fig:asymptotic} we represent the asymptotic state of the two-bump initial data constructed explicitly for small times in \Cref{sec:two deltas}. We recall that why the computations in \Cref{sec:two deltas} are only valid for small time is that the first bump reaches the second bump, and we did not compute the first shock after this happens. However, as we see in \Cref{fig:asymptotic}, the first bump "eats" the second bump (possibly in finite time), and we recover the vortex. Since $u_0(0) = 0$, we have that $u(t,0) = 0$ so the vortex cannot be reached in the supremum norm. Notice also that if $u_0 (0) \ne 0$, then $u(0,t) = (u_0(0)^{-\alpha} + \alpha t)^{-\frac 1 \alpha}$.  Nevertheless, the simulation suggest convergence in all $L^p$ norms for $1 \le p < \infty$.
\begin{figure}[H]
	\centering
	\includegraphics[width=.49\textwidth]{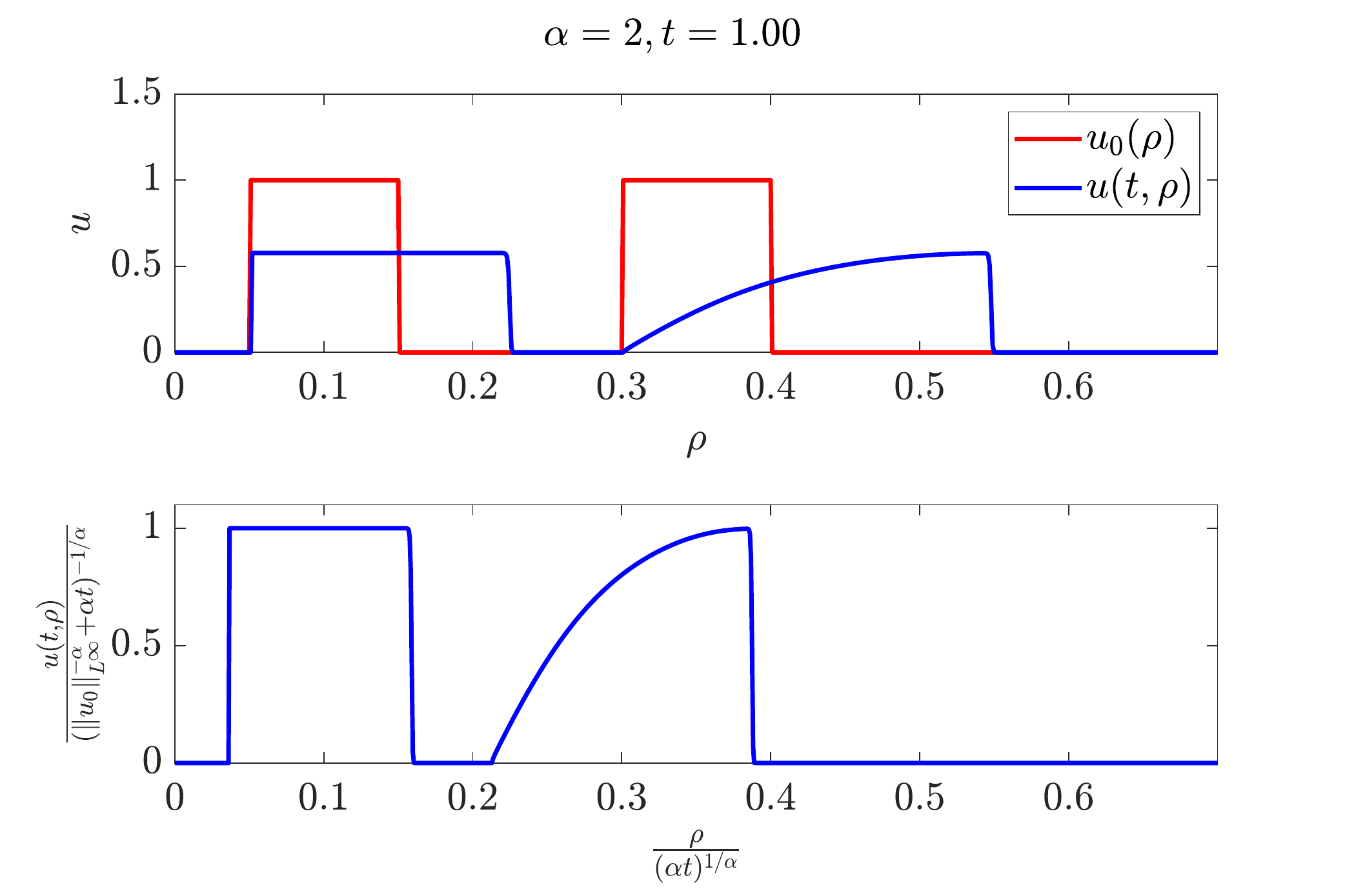}
	\includegraphics[width=.49\textwidth]{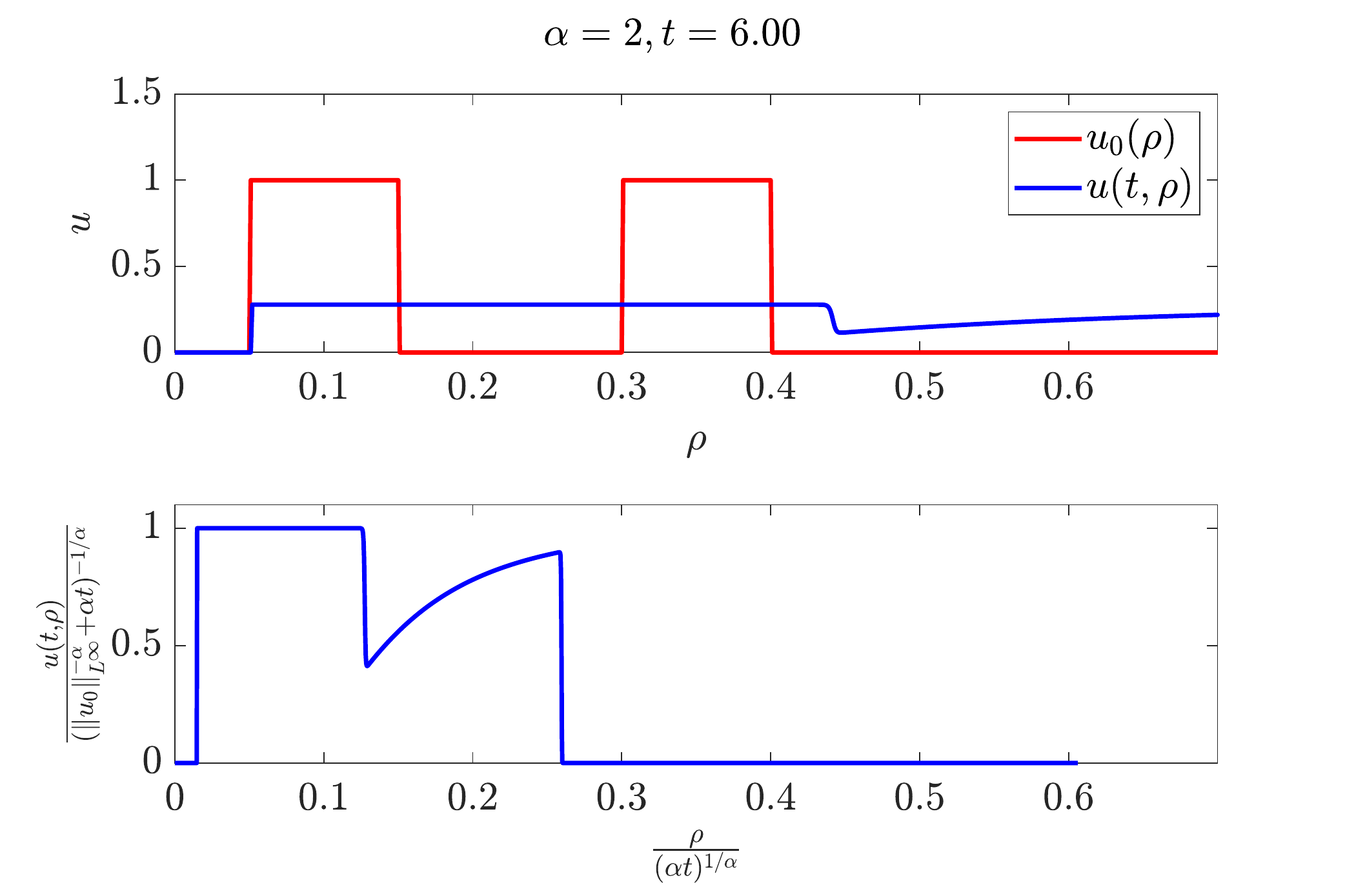}
	
	\includegraphics[width=.49\textwidth]{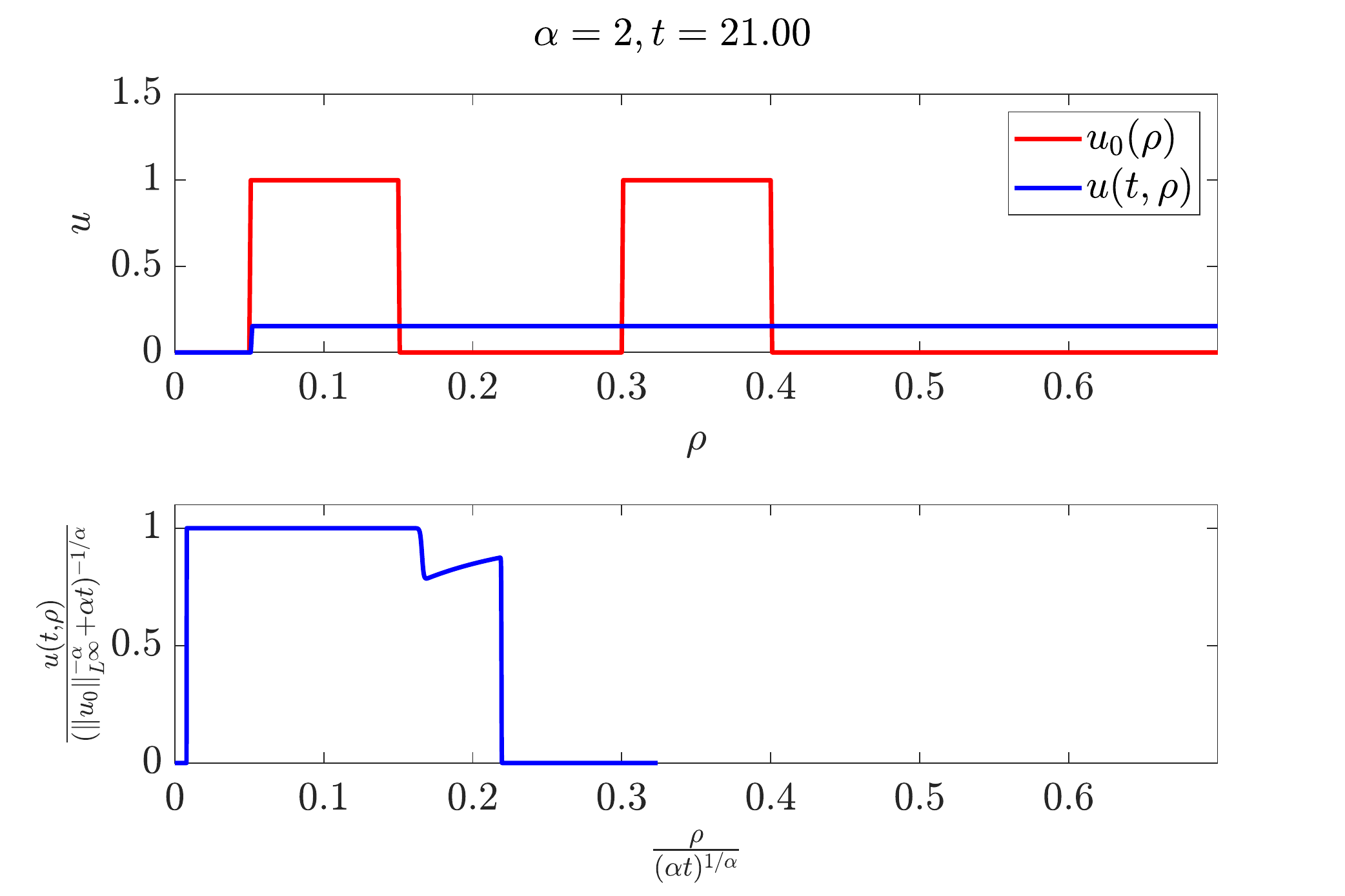}
	\includegraphics[width=.49\textwidth]{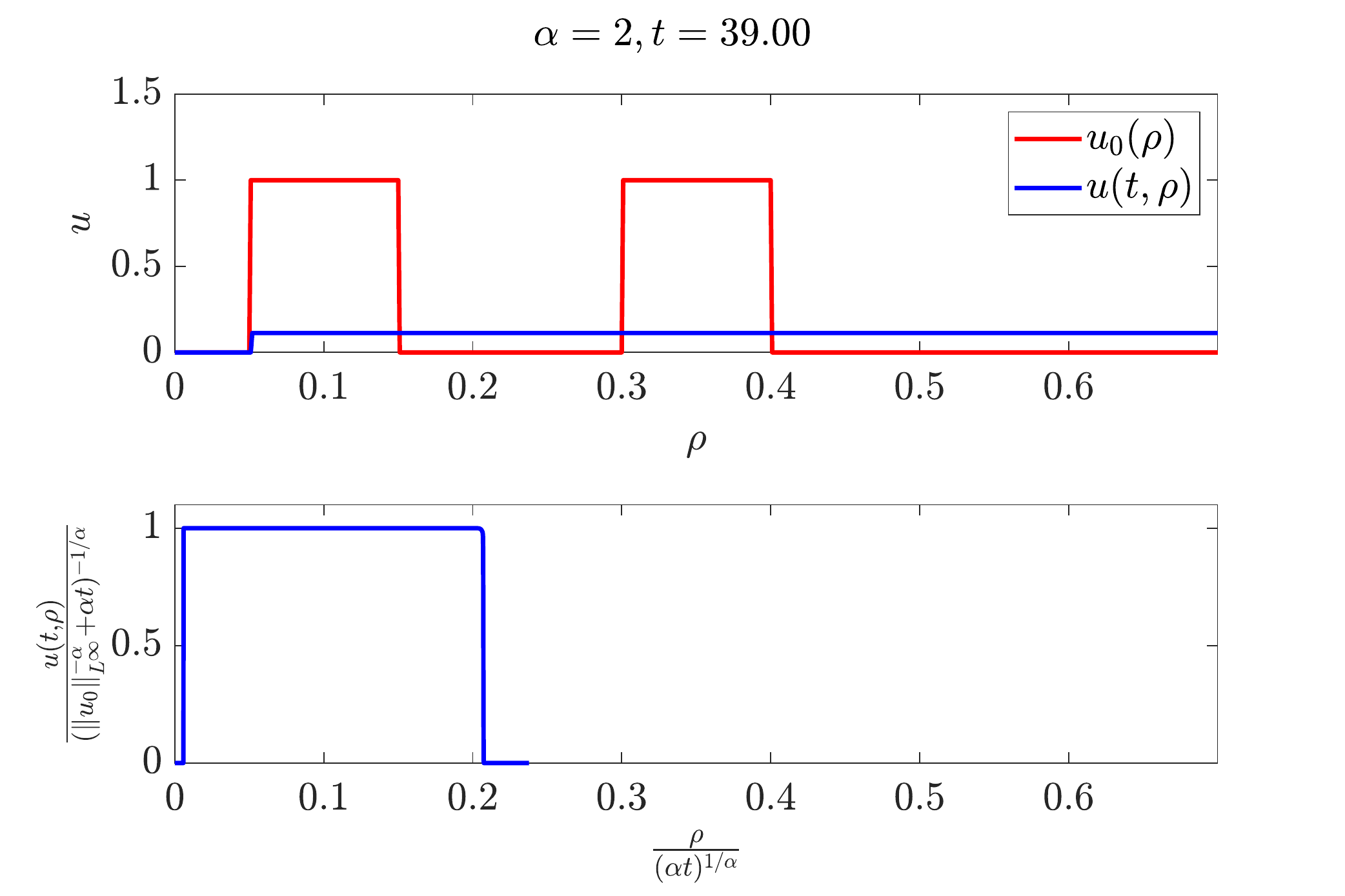}
	
	\caption{Asymptotic behaviour of $u$ in rescaled variables.
		See a movie simulation in the supplementary material \cite[Video 3]{Videos}.}
	\label{fig:asymptotic}
\end{figure}
It is an open problem to determine if the first singularity catches the boundary front in finite time for these particular solutions.

\subsection{Comparison of the waiting time}

It is interesting to compare the behaviour of different powers $u_0(\rho) = (\beta + 1) (1 - \rho)_+^\beta$ which have total mass $M = 1$. There is waiting time if $\beta \ge \frac 1 { \alpha - 1 } $ (see \Cref{thm:waiting time existence} and \Cref{thm:waiting time non-existence}). We will work with $\alpha = 2$. Since the masses are ordered, the waiting time for $u_0 = 2 (1-\rho)_+$ is shorter than that of $u_0 = 3(1-\rho)_+^2.$ It is interesting to notice that the solution for $u_0 = 3(1-\rho)_+^2$ develops a singularity at the interior of the support, before the support starts moving.

\begin{figure}[H]
	\centering
	\includegraphics[width=.32\textwidth]{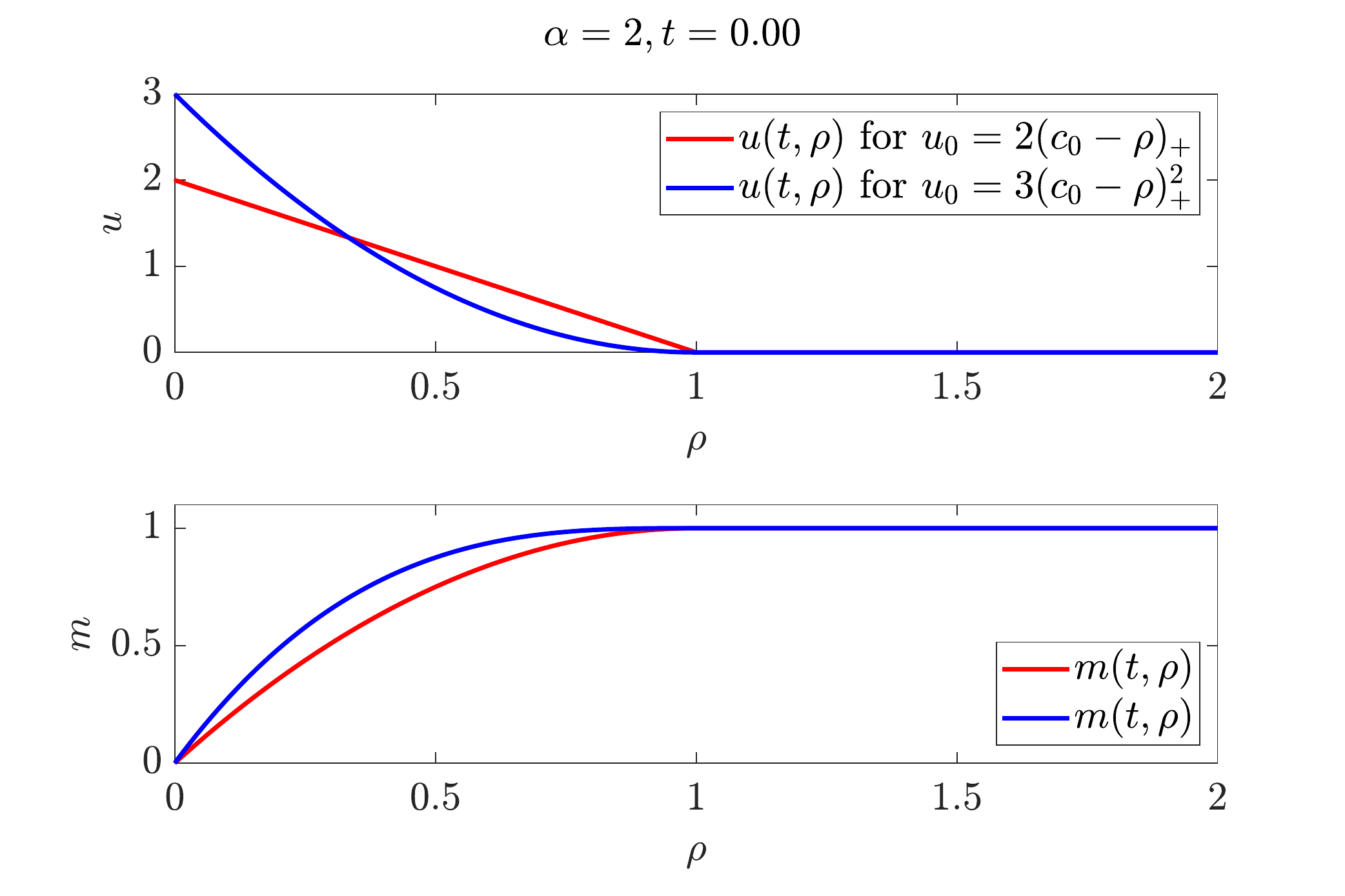}
	\includegraphics[width=.32\textwidth]{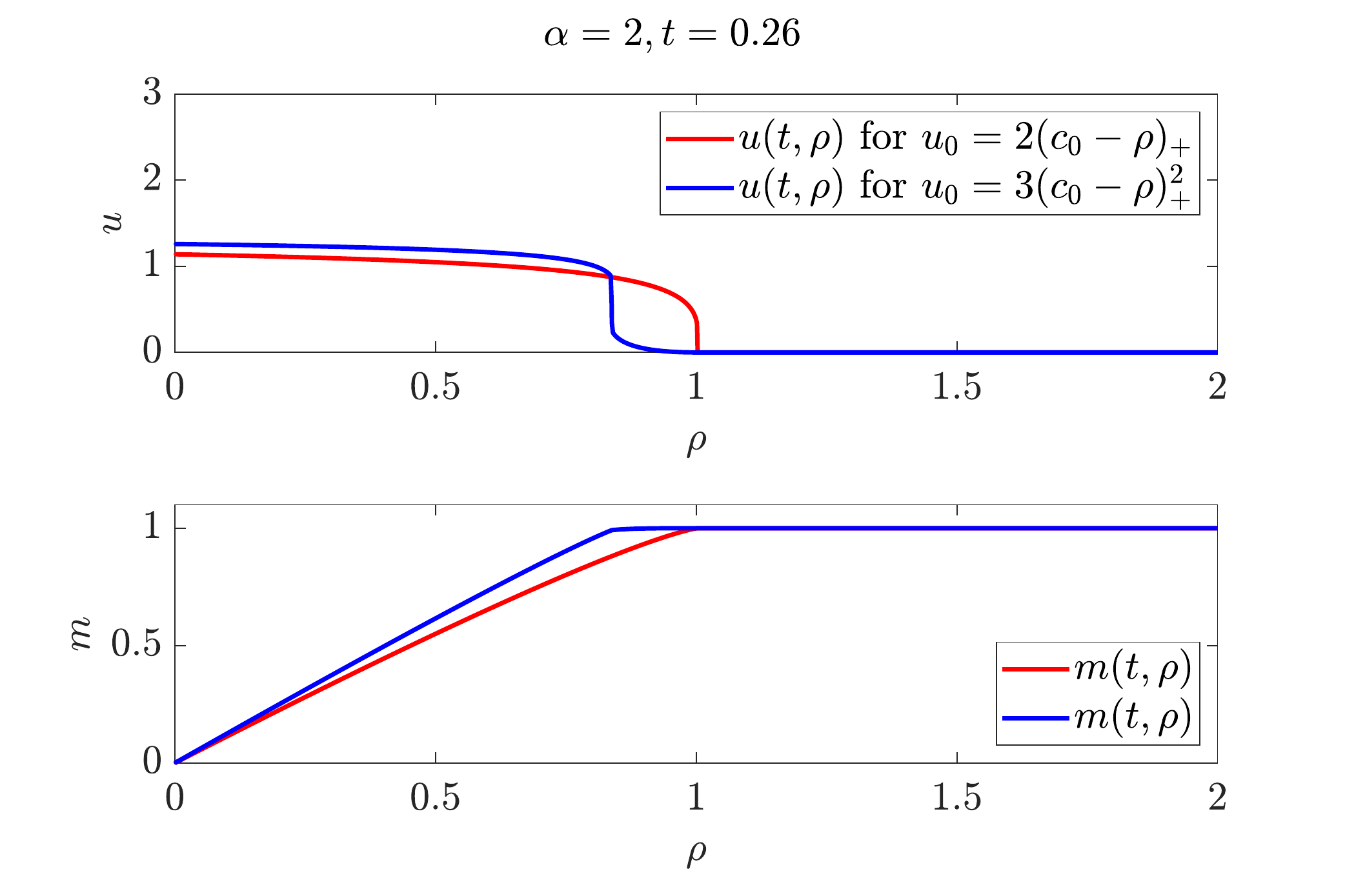}
	\includegraphics[width=.32\textwidth]{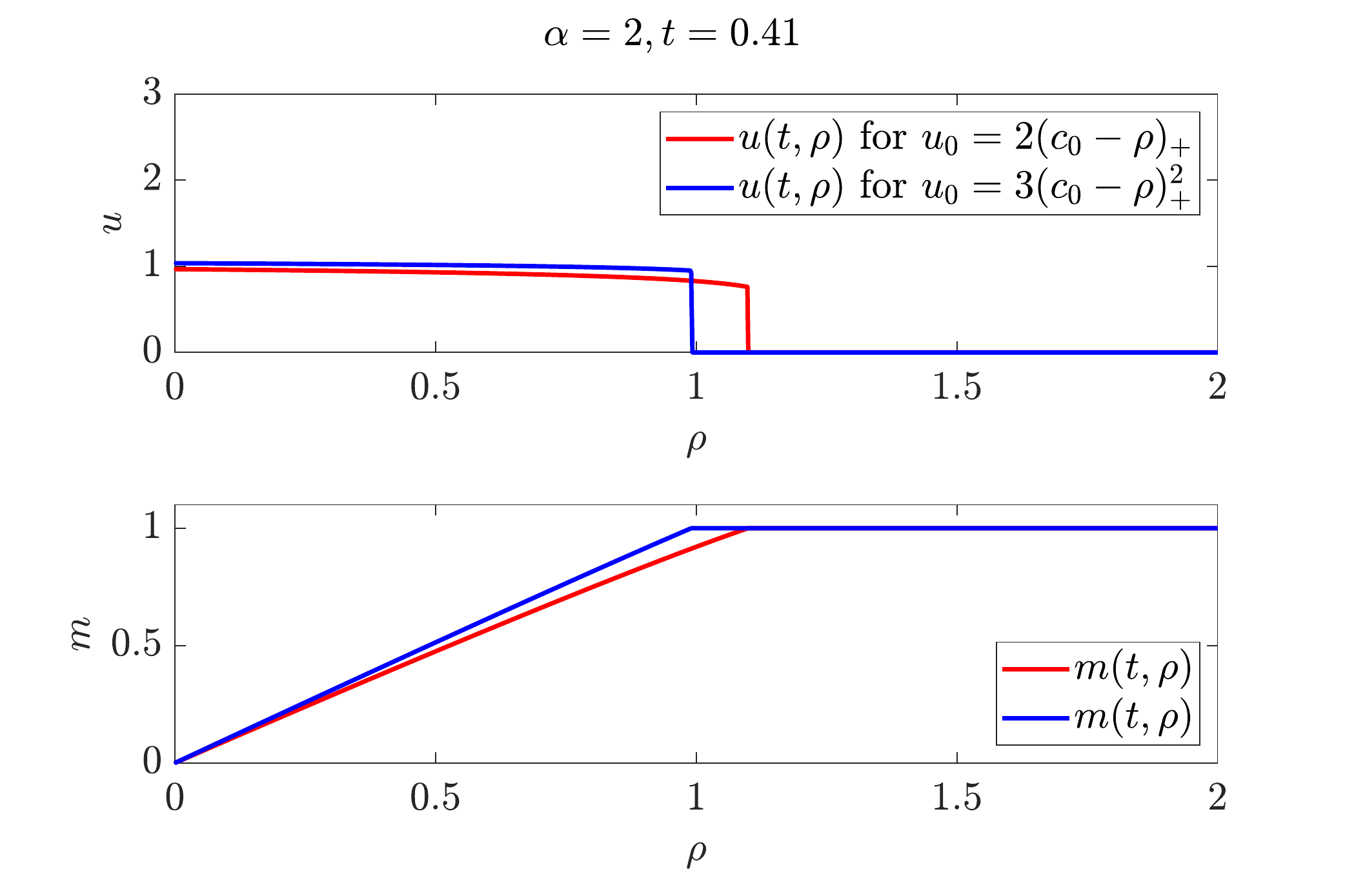}
	\caption{Behaviour of two different powers with waiting time.
		See a movie simulation in the supplementary material \cite[Video 4]{Videos}.}
\end{figure}

\subsection{Level sets of a solution with and without waiting time}

In \Cref{sec:asymptotic behaviour} we showed that $t^{\frac 1 \alpha}$ is the asymptotic behaviour of the support of $u$ for compactly supported $u_0$. For instance, if $u_0$ is a Dirac $\delta$ function at $S(0)$ of mass $M$, we have shown that the support is $[S(0), S(0) + M (\alpha t)^{\frac 1\alpha} ]$. However, for solutions with waiting time, we do not know what is the behaviour of the support for $t$ small. We illustrate an example when $u_0 = (1-\rho)_+$ for $\alpha = 2$ in \Cref{fig:triangle-contour} (cf. \Cref{fig:comparison}).
This initial datum produces a solution with waiting time due to \Cref{thm:waiting time existence}, which by \Cref{thm:solution by characteristics} is initially given by the generalised characteristics. However, as pointed out in \Cref{rem:characteristics not level sets} the characteristics are not the level sets of $m$. Notice that the level sets of $m$ are not straight even for $t$ small.
For comparison, we show a solution not given by characteristics (therefore not a classical solution) and without waiting time (by  \Cref{thm:waiting time non-existence}) which we represent in \Cref{fig:non-waiting-contour}.
\begin{figure}[H]
	\centering
	\includegraphics[width=.7\textwidth]{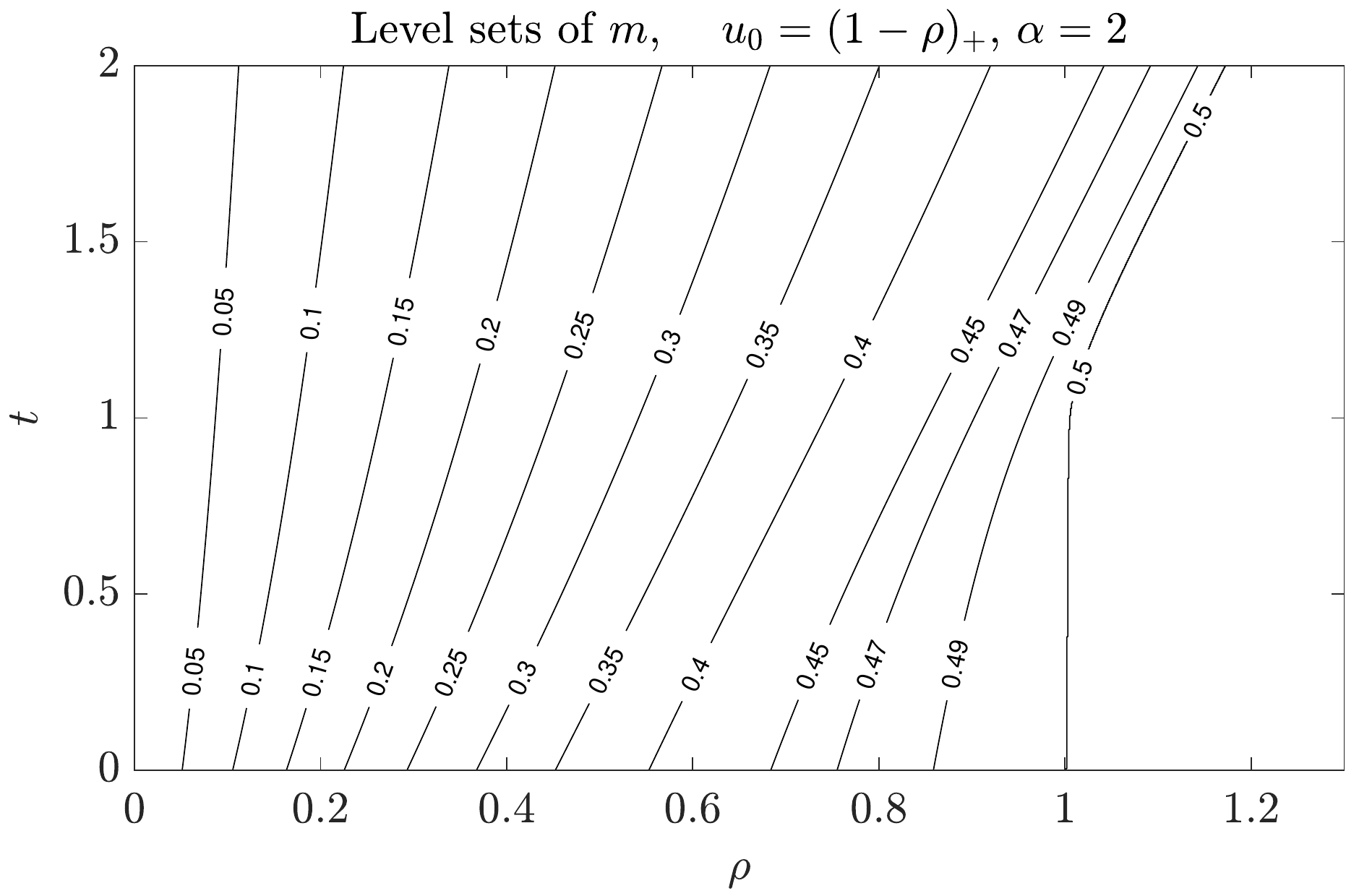}
	\caption{Level sets of the numerical solution with $u_0 = (1-\rho)_+$ for $\alpha = 2$, and a uniform mesh in space of equispaced grid $h_\rho = 1e-3$. In \Cref{fig:comparison} the reader may find a comparison with the mass subsolution with explicit Ansatz.}
	\label{fig:triangle-contour}
\end{figure}
\begin{figure}[H]
	\centering
	\includegraphics[width=.7\textwidth]{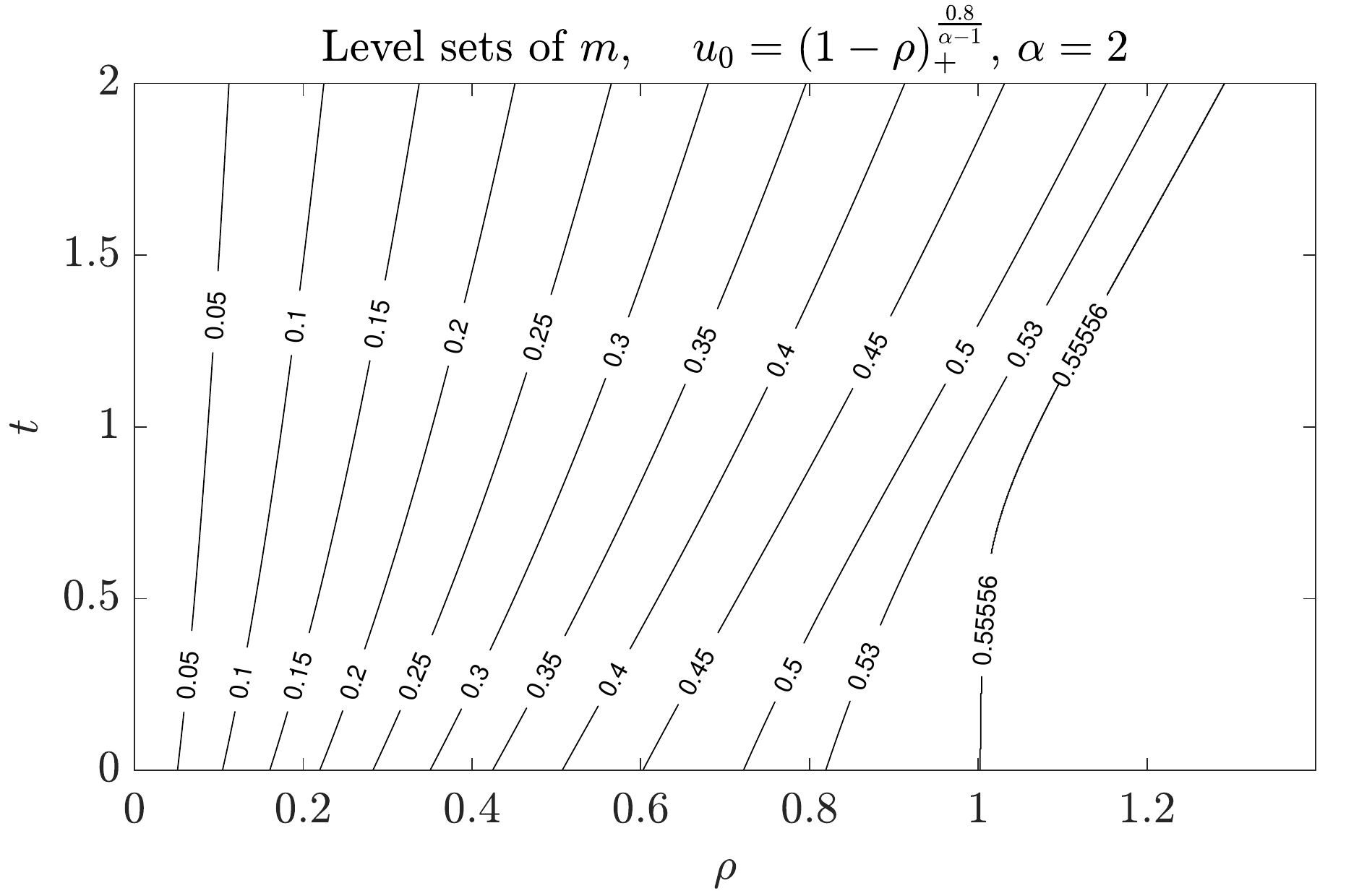}
	\caption{Level sets of the numerical solution not given by characteristics presented in \Cref{rem:no sol by characteristics}, which does not have waiting time due to \Cref{thm:waiting time non-existence}.}
	\label{fig:non-waiting-contour}
\end{figure}

\section*{Remarks and open problems}

\begin{enumerate}
	\item We have constructed a theory of radial solutions and proved well-posedness
 of the mass formulation. Uniqueness in terms of the $u$ variable is an open problem.

 \item Is there a non-radial theory? This seems to be a very difficult problem.
	\item Is there asymptotic convergence to the vortex solution in the $u$ variable in general?
	\item An interesting problem is to construct a theory for infinite mass solutions.
	\item In the two bump solution, is there actually convergence to the vortex in finite time? The numerical experiments suggest so. The ODE for $S_1$ can be written explicitly from the Rankine-Hugoniot condition, and the question is whether $S_1(t) = S_2(t)$ for some $t > 0$.
\end{enumerate}

\section* {Acknowledgments}

JAC was partially supported by EPSRC grant number EP/P031587/1. The research of JAC and DGC was supported by the Advanced Grant Nonlocal-CPD (Nonlocal PDEs for Complex Particle Dynamics:
Phase Transitions, Patterns and Synchronization) of the European Research Council Executive Agency (ERC) under the European Union’s Horizon 2020 research and innovation programme (grant agreement No. 883363).
The research of  DGC and JLV was partially supported by grant PGC2018-098440-B-I00 from the Ministerio de Ciencia, Innovación y Universidades of the Spanish Government. JLV was an Honorary Professor at Univ.\ Complutense.

\printbibliography
\end{document}